\documentclass[12pt]{article}
\usepackage{ayv_paperstyle}
\newcommand{\zd}{\mathbb{Z}^{d}}
\newcommand{\tdr}{\td\times\rd}
\newcommand{\LNr}{\Lambda_N\times\rd}
\newcommand{\LNkr}{\Lambda_{N_k}\times\rd}
\newcommand{\LNB}[1]{\Lambda_N\times\overline{\mathbb{B}}_{#1}}
\title{Continuous time Markov chain based\\ approximation of stationary and weak KAM Hamilton-Jacobi equations}
\author{Yurii Averboukh}
\date{}

\AtEndDocument{
	\vspace{10pt}
	\begin{tabular}{ll}
		Yurii Averboukh:  &
		Krasovskii Institute of Mathematics and Mechanics,  \\ & 
		16 S. Kovalevskaya str.,  Yekaterinburg, Russia; \\
		&
		HSE University,  \\ & 
		11 Pokrovsky Bulvar,  Moscow, Russia;
		\\ & 
		\email{ayv@imm.uran.ru}
\end{tabular}}

\begin{document}
	
	\maketitle
\begin{abstract}
	Main objects of the paper are stationary and weak KAM Hamilton-Jacobi equations on the finite-dimensional torus. 
	The key idea of the paper is to replace the underlying calculus of variations problems with continuous time Markov decision problems. This directly leads to an approximation of the stationary Hamilton-Jacobi equation by the Bellman equation for a discounting Markov decision problem.
	 Developing elements of the weak KAM theory for the Markov decision problem, we obtain an approximation of the effective Hamiltonian. Additionally,  convergences of the functional parts of the discrete weak KAM equations and  Mather measures are shown. It turns out that the approximating equations are systems of algebraic equations. Thus, the paper's result can be seen as  numerical schemes for stationary and weak KAM Hamilton-Jacobi equations. 
	\msccode{49L25, 37J51, 60J28, 93E20, 70H09}
\end{abstract}

\section{Introduction}
\paragraph{Overview of the main results.}
The paper is concerned with the discrete approximations of solutions of stationary Hamilton-Jacobi equation
\begin{equation}\label{intro:eq:lambda_HJ}
	\lambda u_\lambda+H(x,-\nabla u_\lambda)=0
\end{equation} and the weak KAM equation 
\begin{equation}\label{intro:eq:weak_KAM}
	H(x,-\nabla u)=\bar{H}.
\end{equation} 
The Hamiltonian $H$ is assumed to be periodical in $x$, i.e., we consider them on the $d$-dimensional torus $\td$. Notice that in the weak KAM equation, the unknowns are a function $u:\td\rightarrow\mathbb{R}$ and a constant $\bar{H}$.

The stationary Hamilton-Jacobi equation describes the value of the discounting infinite horizon problem of calculus of variations:
\begin{equation}\label{intro:criterion:stationary}
	\text{minimize }\int_{0}^{+\infty}e^{-\lambda s}L(x(s),\dot{x}(s))ds
\end{equation}	
subject to 
\begin{equation}\label{intro:condition:stationary} x(\cdot)\in \operatorname{AC}([0,+\infty);\td),\ \ x(t)=x_*.\end{equation} 
In this case (see~\cite{Fleming2006, Bardi2008c}), the value function of   problem \eqref{intro:criterion:stationary}, \eqref{intro:condition:stationary} is equal to $u_\lambda(x_*)$, where $u_\lambda$ is a viscosity solution of~\eqref{intro:eq:lambda_HJ}. 

Similarly,  weak KAM equation \eqref{intro:eq:weak_KAM} provides~\cite{Fathi2012, Biryuk2010} the fixed points of the backward Lax-Oleinik operator. This means that a pair $(u,\bar{H})$ is a viscosity solution of \eqref{intro:eq:weak_KAM} if and only if, for each $T>0$,
\begin{equation*}\label{intro:problem:H_T}\begin{split}
	u(x_*)=\min\Bigg\{\int_0^TL(x(t),&\dot{x}(t))dt+u(x(T)):\\ &x(\cdot)\in\operatorname{AC}([0,T];\td),\, x(0)=x_*\Bigg\}+\bar{H}T.\end{split}\end{equation*} The number $\bar{H}$ is called an effective Hamiltonian or a Ma\~n\'e critical value.

Our approach relies on the approximation of an absolutely continuous curve $x(\cdot)$ by a continuous-time Markov chain on the regular lattice $\Lambda_N\triangleq \mathbb{Z}^d/(N^{-1}\mathbb{Z}^d)$ with the generator $Q^N(v)=(Q_{x,y}^N(v))_{x,y\in\Lambda_N}$ that depends on a vector $v=(v_1,\ldots,v_d)^T$ and is defined by the rule: 
\begin{equation*}\label{intro:intro:Q}
	Q_{x,y}^N(v)\triangleq \left\{
	\begin{array}{ll}
	 N|v_i|,& y=x+h\operatorname{sgn}(v_i)e_i,\\
	 -N\sum_{j=1}^N|v_j|, & y=x,\\
	 0, &\text{otherwise}.
	\end{array}
	\right.
\end{equation*}  Hereinafter,  $h=N^{-1}$, $\operatorname{sgn}(a)$ denotes the sign of the number $a$, whilst $e_i$ stands for the $i$-th coordinate vector.

For this Markov chain, we consider the criterion
\[\text{minimize }\mathbb{E}\int_{0}^{+\infty}e^{-\lambda t }L(X_t,V_t)dt,  \] where $X_t$ and $V_t$ are a state and a stochastic control respectively. The corresponding Bellman equation turns out to be a system of algebraic equations. 

The first main result of the paper is the rate of approximation of a solution to stationary Hamilton-Jacobi equation~\eqref{intro:eq:lambda_HJ} by the Bellman equation for the aforementioned discounting Markov decision problem. We show that it is of order $N^{-1/2}$ (see Theorem~\ref{rate_lambda:th:rate}). 

The study of the limit behavior of the discounting  Markov decision problem in the case where $\lambda\rightarrow 0$ makes it possible to develop some elements of the weak KAM theory for continuous time Markov chains. In particular, we derive  the weak KAM equation on the lattice that is also a system of algebraic equations. As for the continuous phase space case, we are seeking for a function and  an effective Hamiltonian. 

The second main result is the fact that the effective Hamiltonian for the lattice $\Lambda_N$ approximates the number $\bar{H}$ with an error of the order ${N^{-1/2}}$ (see Theorem~\ref{main_res:th:approx_weak_KAM}). Additionally, the functional parts of the discrete weak KAM equations converge up to subsequence to the functional part of the solution of  weak KAM equation~\eqref{intro:eq:weak_KAM} (see Theorem~\ref{approx_weak_KAM:th:function}).

Recall that, in the continuous phase space  case, the effective Hamiltonian can be characterized via Mather measures those minimize the action $\nu\mapsto \int_{\tdr} L(x,v)\nu(d(x,v))$, where $L$ stands for the Lagrangian. We also introduce the concept of Mather measures for the weak KAM theory on the lattices and show their convergence up to subsequence to a Mather measure for the classical weak KAM theory (see Theorems~\ref{mather:th:mather_ex},~\ref{mather:th:mather_limit}).

\paragraph{Literature overview.}

The modern theory of first order  Hamilton-Jacobi equations relies on the notion of viscosity solutions proposed by Crandall and Lions~\cite{Crandall1983}. This concept, in particular, provides a characterization of the value of  optimal control problems  through  Bellman equations. We refer  to~\cite{Bardi2008c,Fleming2006} for the exposition of the viscosity equation and its applications to the optimal control theory. 

The passing to the limit  in the stationary Hamilton-Jacobi equation as the discounting factor tends to the zero leads to the weak KAM equation~\cite{lions_et_al}. This result is primary used within the study of homogenization of the evolutionary Hamilton-Jacobi equation~\cite{lions_et_al,Evans2002a}. Another source of weak KAM equation is the weak KAM theory itself~\cite{fathi2008weak,Sorrentino}. Recall that this theory studies the long time behavior of the calculus of variation problems and the Euler-Lagrange flows. In particular, the infinitesimal form of the definition of calibrated curve immediately yields the weak KAM  equation~\cite{fathi2008weak,Fathi2012,Evans2008,Evans2004,Biryuk2010}.  Notice that the effective Hamiltonian also known as the Ma\~n\'e critical value provides the averaged optimal outcome for the long time calculus of variations problem. The effective Hamiltonian can be calculated using a Mather measure~\cite{Mather1991} (see also~\cite{Biryuk2010,Sorrentino,Evans2002,fathi2008weak}) that is a measure on the tangent bundle minimizing the action of the Lagrangian over the set of measure invariant w.r.t. the Euler-Lagrange flow (the equivalent form of this condition is obtained in~\cite{Bernard2008,Mane1992}).

There are several extensions and analogs of the weak KAM theory (see, in particular,~\cite{Arnaud2023,Evans2001,Gomes2005,Gangbo2010}). We especially mention papers~\cite{Gomes2002,Soga2021,Iturriaga2005,Mitake2017} where  stochastic analogs of the weak KAM theory are developed. Papers~\cite{Gomes2002,Iturriaga2005,Mitake2017} deal with a problem coming from perturbation of the dynamics by the Brownian motion and give an insights into the selection problem for  Mather measures, whilst in~\cite{Soga2021} the weak KAM theory for the discrete time random walk on the regular lattice is constructed. There, in fact, it is shown that the corresponding weak KAM equations converge to a solution of the continuous phase space weak KAM equation if the time and space steps vanish. More precisely, effective Hamiltonians converge to the continuous phase space effective Hamiltonian, while the function parts converge up to the choice of a subsequence. Thus, one can regard  the results of~\cite{Soga2021} as a numerical method for  weak KAM equation~\eqref{intro:eq:weak_KAM}. This concept is close to one of the paper. However, in the contrast to~\cite{Soga2021}, we start with the continuous time Markov chain. This gives half less number of algebraic equation needed to assure the same order of approximation. Moreover, we impose milder conditions and do not use the completeness of the Euler-Lagrange flow.

Finally, we refer to~\cite{Averboukh2016}, where an approximation of an optimal control problem by a continuous time Markov decision problem was developed for the case of compact    control space. Additionally, the approximation technique for the Hamilton-Jacobi equation based on discrete-time Markov chains was examined in~\cite{Soga2020}.

\paragraph{Organization of the paper.} In Section~\ref{sect:prel}, we give the general notation and assumptions. Additionally, here we remind the definitions of viscosity solutions for equations~\eqref{intro:eq:lambda_HJ},~\eqref{intro:eq:weak_KAM}. 
The next section (Section \ref{sect:markov}) is concerned with the controlled continuous-time Markov chain defined on a regular lattice that plays the crucial role in the paper. We recall the required concepts from the theory of stochastic processes and compute the Hamiltonian for corresponding Markov decision problems. Furthermore,  we evaluate the $L^2$-distance between an absolutely continuous curve and a controlled Markov chains generated by a stochastic strategy. 
The Bellman equation for the discounting Markov decision problem and the existence of an optimal strategy are discussed in Section~\ref{sect:lambda}. 
Then (see Section~\ref{sect:rate_lambda}), we prove the approximation result for the stationary  Hamilton-Jacobi equation. 
Elements of the weak KAM theory for the continuous time Markov decision problem on the lattice are developed in Section~\ref{sect:weak_KAM}. Below, we obtain the  convergence of the  weak KAM equations for the Markov decision problem to the weak KAM equation on the torus (see Section~\ref{sect:approx_weak_KAM}). Finally, in Section~\ref{sect:Mather} we introduce the concept of Mather measure for the Markov decision problem on the lattice and derive the convergence of the Mather measures on regular lattices to a continuous phase space Mather measure.


\section{Preliminaries}\label{sect:prel}
\subsection{General notation and assumptions}
If $(X,\rho_X)$ is a metric space,  $x\in X$, $r>0$, then $\mathbb{B}_r(x)$ stands for the open ball in $X$ of radius $r$ centered at $x$,  i.e.,
\[\mathbb{B}_r(x)\triangleq \big\{y\in X:\, \rho_X(x,y)< r\big\}.\]

	Let $\td\triangleq \rd/\zd$  denote the $d$-dimensional flat torus. Recall that an element of $\td$ is a set given by the rule:
\[x=\{\tilde{x}+n:n\in\zd\}\] for some $\tilde{x}\in\rd$. We denote the standard Euclidean norm on $\rd$ by $|\cdot|$. Further, with some abuse of notation, for $x,y\in \td$, we put
\[|x-y|\triangleq \min\{|x'-y'|:\, x'\in x,\, y'\in y\}.\] The quantity $|x-y|$ is a distance on $\td$. 

Elements of $\rd$ are considered as column vectors, while $\rds$ consists of row vectors. Below we regard $\rd$ as the tangent space to $\td$, and $\rds$ as the cotangent space.  

Further, if $x\in\td$, $\varphi$ is a function defined in a neighborhood of $x$ that is differentiable at $x$, then $\nabla\phi(x)$ denotes  the derivative of $\phi$ at $x$. We assume that $\nabla\phi(x)$ is  a row-vector.

Assume that we are given with a function $L:\td\times\rd\rightarrow\mathbb{R}$ that is called a \textit{Lagrangian}. We impose the following condition on the function $L$:
\begin{enumerate}[label=(L\arabic*)]
	\item\label{prel:assumption:c} $L$ and $L_x$ are continuous;
	\item\label{prel:assumption:positive_definite} $L$ is convex;
	\item\label{prel:assumption:super_linear} $L$ satisfies the superlinear growth condition, i.e., for each $a\geq 0$ there exists a constant $g(a)$ such that 
	\[L(x,v)\geq a|v|+g(a).\]
\end{enumerate} 

In the following, we will widely use the notation 
\begin{equation}\label{prel:intro:K_c}
	K(c)\triangleq \sup_{x\in\td,|v|\leq c}|L_x(x,v)|.
\end{equation}

The Hamiltonian $H:\td\times\rds\rightarrow\mathbb{R}$ is defined by the Legendre transform:
\[H(x,p)\triangleq \max_{v\in\rd}\Big[pv-L(x,v)\Big].\] 

\subsection{Stationary Hamilton-Jacobi  and weak KAM equations}	
	  
The first main object of the paper is  stationary Hamilton-Jacobi equation~\eqref{intro:eq:lambda_HJ}.

We consider the concept of viscosity solutions~\cite{Bardi2008c,Crandall1983}. A function $u_\lambda$ is a viscosity solution of~\eqref{intro:eq:lambda_HJ} if, for every point $x\in\td$ and each smooth function $\psi$ defined in some  neighborhood of $x$  such that the mapping $y\mapsto u_\lambda(y)-\psi(y)$ attains
the minimum (respectively, maximum)   at the point $x$, one has  that $ \lambda u_\lambda(x)+H(x,-\nabla\psi(x))\geq 0$ (respectively,  $\lambda  u_\lambda(x)+H(x,-\nabla\psi(x))\leq 0$).

The second object of the paper is  weak KAM equation~\eqref{intro:eq:weak_KAM}.
As above, its solution is considered in the viscosity sense: a pair $(u,\bar{H})$ is a viscosity solution of~\eqref{intro:eq:weak_KAM} provided that, for every point $x\in\td$ and each  function $\psi\in C^1(\mathbb{B}_r(x))$ for some $r>0$ such that the mapping $y\mapsto u(y)-\psi(y)$ attains
the minimum (respectively, maximum)   at the point $x$, one has  that $H(x,-\nabla\psi(x))\geq \bar{H}$ (respectively,  $H(x,-\nabla\psi(x))\leq \bar{H}$).

In~\cite{lions_et_al}, it is shown that, if, for each sufficiently small $\lambda>0$, $u_\lambda$ solves \eqref{intro:eq:lambda_HJ}, then up subsequence $(u_\lambda(x)-u_\lambda(z),-\lambda u_\lambda(x))$ converge uniformly on $\td$ to a pair $(u(x),\bar{H})$ that is a viscosity solution of weak KAM equation~\eqref{intro:eq:weak_KAM}   (here $z$ is a fixed point on $\td$).


The constant $\bar{H}$ can be characterized through a so called Mather measure $\mu$. To introduce this concept, we follow~\cite{Biryuk2010} and put
\begin{equation}\label{prel:intro:weight}
	\mathscr{w}(v)\triangleq \inf_{x\in\td}L(x,v)\vee 1.
\end{equation} Further, we denote by $\mathcal{M}$ the set of probability measures  on $\tdr$ with finite integral of $\mathscr{w}$, i.e.,
\[\mathcal{M}\triangleq \Bigg\{\nu\in\mathcal{P}(\tdr):\int_{\tdr}\mathscr{w}(v)\nu(d(x,v))<\infty.\Bigg\}\]

A Mather measure~\cite{Biryuk2010} for the torus $\td$ is a probability  on $\tdr$ that minimizes the functional
\[\nu\mapsto \int_{\tdr}L(x,v)\nu( dxdv)\] over the set of measures $\nu\in\mathcal{M}$ satisfying the holonomic constraints:
\[\int_{\tdr}\nabla\phi(x)v\nu( d(x,v))=0\] for each $\phi\in C^1(\td)$. It is shown (see, in particular,~\cite{Biryuk2010}) that, if $\mu$ is a Mather measure for $\td$, then 
\[-\bar{H}=\int_{\tdr}L(x,v)\mu( d(x,v)).\]

\section{Controlled continuous time Markov chain}\label{sect:markov}
\subsection{Construction of approximating Markov chain}

The approximation results derived in the paper rely on a construction of  continuous time Markov chains on infinite or finite time interval. To unify the presentation, we will denote a finite or infinite time by interval $\mathcal{I}$, i.e., either $\mathcal{I}=[0,T]$ or $\mathcal{I}=[0,+\infty)$.

 To explain the main idea of construction of the approximation continuous time Markov chain, we  notice that each function $x(\cdot)\in\operatorname{AC}(\mathcal{I},\rd)$ is entirely determined by $x(0)$ and a velocity $v(\cdot)\in L^1(\mathcal{I},\rd)$. We say that a pair $(x(\cdot),v(\cdot))$ is a control process on $\mathcal{I}$ provided that $x(\cdot)\in C(\mathcal{I},\td)$, $v(\cdot)\in L^1_{\operatorname{loc}}(\mathcal{I},\rd)$ and
 \begin{equation*}\label{markov:eq:x_v}
 	\frac{d}{dt}x(t)=v(t).
 \end{equation*} 

In the Introduction, we already briefly described the main concept of approximating Markov chain. We fix
\begin{enumerate}
	\item a natural number $N$;
	\item a regular lattice $\Lambda_N\triangleq (h\mathbb{Z}^d)/\mathbb{Z}^d\subset \td$. 
\end{enumerate} Hereinafter, we denote
\[h=N^{-1}.\] Additionally, as we mentioned above,  $e_i$ stands for the $i$-th coordinate vector on $\rd$. 

The Kolmogorov matrix of the approximating Markov chain is constructed as follows. If $v=(v_1,\ldots,v_d)^T\in\rd$, then we define the matrix 
$Q^N(v)=(Q^N_{x,y}(v))_{x,y\in\Lambda_N}$ by the following rule:
\begin{equation*}\label{markov:intro:Kolmogorov}
	Q^N_{x,y}(v)\triangleq \left\{
	\begin{array}{ll}
		h^{-1}|v_i|, & y=x+\operatorname{sgn}(v_i)he_i,\\
		-h^{-1}\sum_{j=1}^d|v_j|, & y=x,\\
		0, & \text{otherwise}.
	\end{array}
	\right.
\end{equation*} 
Now let us define  stochastic control processes those are determined by this Kolmogorov matrix. First, we start with open-loop controls.
\begin{definition}\label{markov:def:open_loop}
	We say that a 6-tuple $(\Omega,\mathcal{F},\{\mathcal{F}_{t}\}_{t\in \mathcal{I}},\mathbb{P},X,V)$ is a controlled Markov chain process for the lattice $\Lambda_N$ provided that
	\begin{itemize}
		\item $(\Omega,\mathcal{F},\{\mathcal{F}_{t}\}_{t\in \mathcal{I}},\mathbb{P})$ is a filtered probability space;
		\item $X$ is a $\{\mathcal{F}_{t}\}_{t\in \mathcal{I}}$-adopted process with values in $\Lambda_N$;
		\item $V$ is a $\{\mathcal{F}_{t}\}_{t\in \mathcal{I}}$-progressively measurable process with values in $\rd$ such that $\mathbb{E}\int_{0}^t|V_s|dt<\infty$ for each $t\in\mathcal{I}$;
		\item for each $\phi:\Lambda_N\rightarrow \mathbb{R}$ the process
		\[\phi(X_t)-\int_0^t Q^N(V_s)\phi(X_s)ds\] is a $\{\mathcal{F}_t\}_{t\in \mathcal{I}}$-martingale.
	\end{itemize} Hereinafter, $\mathbb{E}$ stands for the expectation corresponding to the probability $\mathbb{P}$.
\end{definition} Below, we denote the set of all controlled Markov chain process  for the lattice $\Lambda_N$ on the time interval~$\mathcal{I}$ satisfying the initial condition $X(0)=x_*$ $\mathbb{P}$-a.s. by $\operatorname{MCP}_N(\mathcal{I},x_*)$. 

The most convenient tool in the theory of controlled Markov chain is feedback strategies (policies).
\begin{definition}\label{markov:def:policy}
	A feedback strategy on $\Lambda_N$ is a measurable mapping $\pi:\mathcal{I}\times\Lambda_N\rightarrow\rd$. A stationary feedback strategy is a mapping $\pi:\Lambda_N\rightarrow\mathbb{R}^d$. 
\end{definition}
A feedback strategy $\pi$ defines a Kolmogorov matrix
$\mathcal{Q}^N[\pi,t]=(\mathcal{Q}_{x,y}^N[\pi,t])_{x,y\in\Lambda_N}$ with entries
\[\mathcal{Q}_{x,y}^N[\pi,t]=Q_{x,y}^N[\pi(t,x)].\] 

If $\pi$ is a stationary feedback strategy, we will write simply $\mathcal{Q}^N[\pi]$ instead of $\mathcal{Q}^N[\pi,t]$.

\begin{definition}\label{markov:def:motion_policy}
Given a feedback strategy $\pi$, we say that $(\Omega,\mathcal{F},\{\mathcal{F}_{t}\}_{t\in \mathcal{I}},\mathbb{P},X)$ is a motion produced by $\pi$ provided that 
\begin{itemize}
	\item $(\Omega,\mathcal{F},\{\mathcal{F}_{t}\}_{t\in \mathcal{I}},\mathbb{P})$ is a filtered probability space;
	\item $X$ is $\{\mathcal{F}_{t}\}_{t\in \mathcal{I}}$-adopted process with values in $\Lambda_N$;
	\item for every $t\in\mathcal{I}$,
	\[\mathbb{E}\int_{0}^t|\pi(s,X_s)|ds<\infty;\]
	\item for each $\phi:\Lambda_N\rightarrow \mathbb{R}$, the process
	\begin{equation}\label{lambda_control:expression:Q_v}
		\phi(X_t)-\int_{0}^t\mathcal{Q}^N[\pi,s]\phi(X_s)ds
	\end{equation} is a $\{\mathcal{F}_{t}\}_{t\in \mathcal{I}}$-martingale.
\end{itemize}
\end{definition} If $(\Omega,\mathcal{F},\{\mathcal{F}_{t}\}_{t\in \mathcal{I}},\mathbb{P},X)$ is a motion produced by $\pi$, then, letting 
\[V_t\triangleq \pi(t,X_t),\]  we obtain a  controlled Markov chain process $(\Omega,\mathcal{F},\{\mathcal{F}_{t}\}_{t\in \mathcal{I}},\mathbb{P},X,V)$.

Further, a distribution on $\Lambda_N$ is a sequence $m=(m_{x})_{x\in\Lambda_N}$ with nonnegative entries such that
\[\sum_{x\in\Lambda_N}m_{x}=1.\] We say that a motion $(\Omega,\mathcal{F},\{\mathcal{F}_{t}\}_{t\in [0,+\infty)},\mathbb{P},X)$ produced by a feedback strategy $\pi$ has the initial distribution $m_0 =\{m_{0,x}\}_{x\in\Lambda_N}$ if 
\[\mathbb{P}(X_0=x)=m_{0,x},\ \ x\in\Lambda_N.\]  Additionally, we say that a motion $(\Omega,\mathcal{F},\{\mathcal{F}_{t}\}_{t\in [0,+\infty)},\mathbb{P},X)$ produced by $\pi$ has the initial state $z$ provided that 
\[X(0)=z,\ \ \mathbb{P}-\text{a.s.} \] This corresponds to the initial distribution $\mathbbm{1}_z=(\mathbbm{1}_{z,x})_{x\in\Lambda_N}$ with entries
\[
\mathbbm{1}_{z,x}\triangleq \left\{\begin{array}{cc}
	1, & x=z,\\
	0,& x\neq z.
\end{array}\right.
\]

There exists (see \cite[Theorem 5.4.1]{Kolokoltsov}) at least one motion produced by the stationary feedback strategy  $\pi$ with the initial distribution equal to $m_0$.

Notice that an evolution of  distributions produced by the feedback strategy $\pi$ and the initial distribution $m_0$ is described by a mapping $\mathcal{I}\ni t\mapsto m(t)=(m_x(t))_{x\in\Lambda}$ such that 
\begin{equation}\label{markov:property:m}
	m_{x}(t)=\mathbb{P}(X_t=x).
\end{equation}
The function $ m(\cdot)$ satisfies by the Kolmogorov equation 
\begin{equation}\label{markov:eq:Kolmogorov}
	\frac{d}{dt}m(t)=m(t)\mathcal{Q}^N[\pi,t],\ \ m(t)=m_0.
\end{equation} Here, we regard $m(t)$  as a row-vector.

\subsection{Finite differences}
Let $\phi:\Lambda_N\rightarrow\mathbb{R}$, $x\in\Lambda_N$ and $i\in \{1,\ldots, d\}$. Define, for $x\in\Lambda_N$,
\begin{equation}\label{markov:intro:Deltas_i}
	\begin{split}
		&\Delta_{N,i}^+\phi(x)\triangleq \frac{\phi(x+he_i)-\phi(x)}{h},\\
		&\Delta_{N,i}^-\phi(x)\triangleq \frac{\phi(x-he_i)-\phi(x)}{h}.
	\end{split}
\end{equation} If there exists a smooth function $u$ such that $\phi(x)=u(x)$ on $\Lambda_N$, then the quantity $\Delta_{N,i}^+\phi(x)$ is the standard right  approximation of the partial derivative of $u$  at $x$ w.r.t. $x_i$. Analogously, $\Delta_{N,i}^-\phi(x)$ is a left approximation of the $-\partial_{x_i}u(x)$.

Notice that 
\begin{equation}\label{markov:equality:Delta_shift}
	\begin{split}
		&\Delta_{N,i}^+\phi(x)=-\Delta_{N,i}^-\phi(x+he_i),\\
		&\Delta_{N,i}^-\phi(x)=-\Delta_{N,i}^+\phi(x-he_i).
	\end{split}
\end{equation}

Further, we set 
\begin{equation*}
	\Delta_{N}\phi(x)=((\Delta_{N,1}^+\phi(x),\Delta_{N,1}^-\phi(x)),\ldots, (\Delta_{N,d}^+\phi(x),\Delta_{N,d}^-\phi(x))).
\end{equation*} Thus, it is convenient to work with  sequences of pairs. 
If $\xi=((\xi_1^+,\xi_1^-),\ldots, (\xi_d^+,\xi_d^-))$ is a $d$-tuple of pairs, where $\xi_i^+,\xi_i^-\in\mathbb{R}$, $v\in \rd$, then, with some abuse of notation, we denote
\begin{equation*}
	\xi\cdot v\triangleq \sum_{i=1}^{d}\big[\xi_i^+ v_i^++\xi_i^i v_i^-\big].
\end{equation*}
Notice that, for each $\phi:\Lambda_N\rightarrow\mathbb{R}$ and $x\in\Lambda_N$, one has that
\begin{equation}\label{markov:eq:Q_Delta}
	\sum_{y\in\Lambda_N}Q_{x,y}^N[v]\phi(y)=\Delta_{N}\phi(x)\cdot v.
\end{equation} 
Thus, if   $(\Omega,\mathcal{F},\{\mathcal{F}_{t}\}_{t\in \mathcal{I}},\mathbb{P},X,V)$ is an controlled Markov chain process for the lattice $\Lambda_N$, then, for each function $\phi:\Lambda_N\rightarrow\mathbb{R}$, and $s,r\in\mathcal{I}$, $s<r$,
\[\mathbb{E}\phi(X_r)-\mathbb{E}\phi(X_s)=\mathbb{E}\int_s^r\big(\Delta_{N}\phi(X_t)\cdot V_t\big)dt.\] In particular, given a feedback strategy $\pi$ and a corresponding motion $(\Omega,\mathcal{F},\{\mathcal{F}_{t}\}_{t\in \mathcal{I}},\mathbb{P},X)$, we have that 
\[\mathbb{E}\phi(X_r)-\mathbb{E}\phi(X_s)=\mathbb{E}\int_s^r\big(\Delta_{N}\phi(X_t)\cdot \pi(t,X_t)\big)dt.\]

Now, for $x\in\Lambda_N$, $\xi= ((\xi_i^+,\xi_i^-))_{i=1}^d$, where $\xi_i^+,\xi_i^-\in\rd$, we set
\begin{equation}\label{prel:intro:H_lattice}
	\mathcal{H}_N(x,\xi)\triangleq \sup_{v\in\rd}\Big[\xi\cdot v-L(x,v)\Big].
\end{equation} 
This and \eqref{markov:eq:Q_Delta} give that 
\begin{equation}\label{prel:equality:Ham}
	\begin{split}
		\mathcal{H}_N(x,(-\Delta_{N})\phi(x))&=\max_{v\in\rd}\Big[(-\Delta_{N})\phi(x)\cdot v-L(x,v)\Big]\\&=-\min_{v\in\rd}\Bigg[\sum_{y\in\Lambda_N} Q^N_{x,y}[v]\phi(y)+L(x,v)\Bigg].
	\end{split}
\end{equation} 
Here, \[(-\Delta_{N})\phi(x)=(((-\Delta_{N,1}^+)\phi(x),(-\Delta_{N,1}^-)\phi(x)),\ldots,((-\Delta_{N,d}^+)\phi(x),(-\Delta_{N,d}^-)\phi(x))).\]

\subsection{Distance between the deterministic evolution and the Markov chain}

We recall that $\td=\rd/\mathbb{Z}^d$, while $\Lambda_N=(h\mathbb{Z}^d)/\mathbb{Z}^d$. It is convenient to denote by $[\tilde{x}]$ the equivalence class of $\tilde{x}\in\rd$, i.e.,
\[[\tilde{x}]\triangleq \{\tilde{x}+n:n\in\mathbb{Z}^d\}.\]

Let $\pi$ be a feedback strategy on $\Lambda_N$.  We assume that its entries are uniformly bounded. Denote by $\hat{\pi}$  a feedback strategy on $h\mathbb{Z}^d$ such that, for $\tilde{x}\in h\mathbb{Z}^d$, $\hat{\pi}(t,\tilde{x})=\pi(t,[\tilde{x}])$.

Now we define the  infinite matrix $\widehat{\mathcal{Q}}^N[\pi,t]=\{\widehat{\mathcal{Q}}^N_{\tilde{x},\tilde{y}}[\pi,t]\}_{\tilde{x},\tilde{y}\in h\mathbb{Z}^d}$ indexed with elements of $h\mathbb{Z}^d$ by the rule:
\[\widehat{\mathcal{Q}}^N_{\tilde{x},\tilde{y}}[\pi,t]\triangleq \left\{
\begin{array}{ll}
	h^{-1}|\hat\pi_i(t,\tilde{x})|, & \tilde y=\tilde x+h\operatorname{sgn}(\hat{\pi}_i(t,\tilde{x}))e_i,\\
	-h^{-1}\sum_{j=1}^d |\hat\pi_j(t,\tilde{x})|, & \tilde{y}=\tilde{x},\\
	0, & \text{otherwise}.
\end{array}
\right. \] Notice that, if $\tilde{y}=\tilde{x}\pm he_i$, then $\widehat{\mathcal{Q}}^N_{\tilde{x},\tilde{y}}[\pi,t]=\mathcal{Q}^N_{x,y}[\pi,t]$, where $x=[\tilde{x}]$, $y=[\tilde{y}]$.

Let us introduce the  generator $\mathcal{L}^\pi_t$ on $\rd\times (h\mathbb{Z}^d)$ by the following rule: for $\phi\in C^1(\rd\times (h\mathbb{Z}^d))$ with at most quadratic growth,
\[\mathcal{L}_t^\pi\phi(\tilde x^1,\tilde x^2)\triangleq \nabla_{\tilde x^1}\phi(\tilde x^1,\tilde x^2)\pi(t,\tilde{x}^2)+\sum_{\tilde{y}\in h\mathbb{Z}^d}\widehat{\mathcal{Q}}^N_{\tilde x^2,\tilde{y}}[\pi,t]\phi(\tilde x^1,\tilde{y}).\] Hereinafter, we denote by $C^1(\rd\times (h\mathbb{Z}^d))$ the set of continuous functions from $\rd\times (h\mathbb{Z}^d)$ to $\mathbb{R}$ those are continuously differentiable w.r.t. the first variable.

It follows from \cite[Theorem 5.4.1]{Kolokoltsov} that, given $(\tilde x^1_*,\tilde x^2_*)\in\rd\times (h\mathbb{Z}^d)$, there exist a filtered probability space $(\Omega,\mathcal{F},\{\mathcal{F}_t\}_{t\in \mathcal{I}},\mathbb{P})$ and $\{\mathcal{F}_t\}_{t\in \mathcal{I}}$ adopted process $(\widetilde{X}^1,\widetilde{X}^2)$ such that 
\begin{equation}\label{markov:eq:tilde_X_ini}
	(\widetilde{X}^1_0,\widetilde{X}^2_0)=(\tilde x^1_*,\tilde x^2_*)
\end{equation} 
and, for each  $\phi\in C^1(\rd\times (h\mathbb{Z}^d))$ with at most quadratic growth, one has that
the process
\begin{equation}\label{markov:quantity:phi_int}
	\phi(\widetilde{X}^1_t,\widetilde{X}^2_t)-\int_0^t \mathcal{L}_s^\pi\phi(\widetilde{X}^1_s,\widetilde{X}^2_s)ds
\end{equation} 
is a $\{\mathcal{F}_t\}_{t\in \mathcal{I}}$-martingale. Notice that, in this case, $\widetilde{X}^1$ satisfies $\mathbb{P}$-a.s. the system of ODEs
\[\frac{d}{dt}\widetilde{X}^1_t=\hat\pi(t,\widetilde{X}^2_t),\] while $\widetilde{X}^2$ is a continuous time Markov chain with Kolmogorov matrix $\widehat{\mathcal{Q}}^N[\pi,\cdot]$.
Furthermore, if we let 
\begin{equation}\label{markov:intro:X}
	X^1_t\triangleq [\widetilde{X}^1_t],\ \ X^2_t\triangleq [\widetilde{X}^2_t], 
\end{equation} 
then 
\[\frac{d}{dt}{X}^1_t=\pi(t,{X}^2_t),\] and $(\Omega,\mathcal{F},\{\mathcal{F}_t\}_{t\in \mathcal{I}},\mathbb{P},X^2)$ is a motion produced by the feedback strategy $\pi$ and the initial  state~$x^2_*=[\tilde{x}^2_*]$. 

The key result of this section is the following.
\begin{lemma}\label{markov:lm:distance} Let 
	\begin{itemize}	
		\item $\pi$ be a feedback strategy such that, for some constant $c>0$, $|\pi(t,x)|\leq c$;
		\item $x^1_*\in\td$, $x^2_*\in \Lambda_N$, 
		\item $\tilde{x}_*^1\in \rd$, $\tilde{x}^2_*\in h\mathbb{Z}^d$ be such $x^1_*=[\tilde{x}^1_*]$, $x^2_*=[\tilde{x}^2_*]$ and $|x_*^1-x^2_*|=|\tilde x_*^1-\tilde x^2_*|$;
		\item $(\Omega,\mathcal{F},\{\mathcal{F}_t\}_{t\in\mathcal{I}},\mathbb{P})$ be a filtered probability space; 
		\item $(\widetilde{X}^1,\widetilde{X}^2)$ be a stochastic process defined on $(\Omega,\mathcal{F},\{\mathcal{F}_t\}_{t\in\mathcal{I}},\mathbb{P})$ with values in $\rd\times (h\mathbb{Z}^d)$ satisfying conditions \eqref{markov:eq:tilde_X_ini} and \eqref{markov:quantity:phi_int};
		\item $(X^1,X^2)$ be defined by \eqref{markov:intro:X}.
	\end{itemize} 
	Then, for each $t\in \mathcal{I}$,
	\begin{equation}\label{markov:ineq:distance}
		\mathbb{E}|{X}^1_t-{X}^2_t|^2\leq\mathbb{E}|\widetilde{X}^1_t-\widetilde{X}^2_t|^2\leq |x^1_*-x^2_*|+\sqrt{d}cN^{-1}t.
	\end{equation} 	
\end{lemma}
\begin{proof}
	The first inequality in \eqref{markov:ineq:distance} directly follows from the construction of $({X}^1,{X}^2)$. 
	
	To prove the second inequality, we first notice that, due to the assumption that $\pi$ has uniformly bounded entries, $\mathbb{E}|\widetilde{X}^1_t-\widetilde{X}^2_t|^2$ is bounded. Further, let $\mathscr{q}(\tilde x_1,\tilde{x}_2)\triangleq |\tilde{x}_1-\tilde{x}_2|^2$.
	Plugging in  \eqref{markov:quantity:phi_int} the function $\mathscr{q}$, we obtain that 
	\[
	\begin{split}
		\mathbb{E}|\widetilde{X}^1_t-\widetilde{X}^2_t|^2\leq  \mathbb{E}|\widetilde{X}^1_0-\widetilde{X}^2_0|^2+\mathbb{E}\int_0^t\mathcal{L}^\pi_s\mathscr{q}(\widetilde{X}^1_s,\widetilde{X}^2_s)ds.
	\end{split}
	\]
	Direct computation gives that 
	\[\mathcal{L}^\pi_t\mathscr{q}(x_1,x_2)=h\sum_{i=1}^d|\pi_i(t,x_2)|\leq hc\sqrt{d}.\] Therefore, taking into account that $h=N^{-1}$, we obtain the second inequality in \eqref{markov:ineq:distance}.
\end{proof}

\section{Discounting Markov decision problem}\label{sect:lambda}
In this section, we work with the infinite time interval, i.e., we put $\mathcal{I}=[0,+\infty)$.

If $(\Omega,\mathcal{F},\{\mathcal{F}_t\}_{t\in [0,+\infty)},\mathbb{P},X,V)$ is a controlled Markov chain process for the lattice $\Lambda_N$, then its quality is evaluated by the quantity
\[\mathbb{E}\int_{0}^{+\infty}e^{-\lambda t}L(X_t,V_t).\] 

Furthermore, the outcome  of the feedback strategy $\pi$ and the initial distribution $m_0$ is equal to
\[\mathcal{J}_{N,\lambda}[\pi,m_0]\triangleq \mathbb{E}\int_{0}^{+\infty}e^{-\lambda t}L(X_t,\pi(t,X_t)dt,\] where $(\Omega,\mathcal{F},\{\mathcal{F}_{t}\}_{t\in [0,+\infty)},\mathbb{P},X)$ is a motion produced by $\pi$ and $m_0$. The quantity $\mathcal{J}_{N,\lambda}[\pi,m_0]$ does not depend on the concrete choice of the process $(\Omega,\mathcal{F},\{\mathcal{F}_{t}\}_{t\in [0,+\infty)},\mathbb{P},X)$ produced by $\pi$ and $m_0$.  Indeed, \eqref{markov:property:m} implies that 
\begin{equation*}\label{lambda:equality:J_N_lambda}
	\mathcal{J}_{N,\lambda}[\pi,m_0]=\int_{0}^{+\infty}e^{-\lambda t}\int_{\td}L(x,\pi(t,x))m(t)dt.
\end{equation*} 
Here $m(\cdot)$ satisfies \eqref{markov:eq:Kolmogorov}. If $m_0$ is equal to $1$ at $z$ and zero elsewhere, then we write $\mathcal{J}_{N,\lambda}[\pi,z]$ instead of $\mathcal{J}_{N,\lambda}[\pi,m_0]$.

The Bellman equation for the examined discounting Markov decision problem takes the form:
\begin{equation}\label{lambda:eq:HJ_Delta}
	\lambda \varphi_{N,\lambda}(x)+\mathcal{H}_N(x,(-\Delta_N)\varphi_{N,\lambda}(x))=0,\ \ x\in\Lambda_N.
\end{equation} This fact is formalized in the  following statement.

\begin{proposition}\label{lambda:prop:existence}
	Equation \eqref{lambda:eq:HJ_Delta} has a unique solution. Moreover, the stationary feedback strategy $\pi_{\lambda,N}^*$ defined by the rule:
	\begin{equation}\label{lambda:incl:pi_star}
		\pi_{\lambda,N}^*(x)\in\underset{v\in \rd}{\operatorname{Argmax}}\Big[(-\Delta_N)\varphi_{N,\lambda}(x)\cdot v-L(x,v)\Big]
	\end{equation} 
is optimal, i.e.,
	\begin{equation}\label{lambda:equality:vaphi_pi_opt}\varphi_{N,\lambda}(z)=\mathcal{J}_{N,\lambda}[\pi_{\lambda,N}^*,z]\end{equation} and, for each  $(\Omega,\mathcal{F},\{\mathcal{F}_{t}\}_{t\in [0,+\infty)},\mathbb{P},X,V)\in\operatorname{MCP}_N([0,+\infty),z)$, one has
	\begin{equation}\label{lambda:ineq:J_vaphi}
		\mathbb{E}\int_0^{+\infty}e^{-\lambda t}L(X_t,V_t)dt\geq \varphi_{N,\lambda}(z).
	\end{equation}
\end{proposition}
\begin{proof}
	The proof relies on truncations arguments. Let $A>0$. We consider the Markov decision problem 
	\[\text{minimize }\mathbb{E}\Bigg[\int_0^{+\infty}e^{-\lambda t}L(X_t,V_t)dt\Bigg]\] over the set of 6-tuples $(\Omega,\mathcal{F},\{\mathcal{F}_{t}\}_{t\in [0,+\infty)},\mathbb{P},X,V)\in\operatorname{MCP}_N([0,+\infty),z)$ satisfying the additional constraint $|V_t|\leq A$. Due to \cite[Lemma 4.4, Theorems 4.6, 4.10]{Guo_Hernandez_Lerma}, this problem has a value denoted by $\varphi_{N,\lambda}^A$. The latter satisfies the following Bellman equation on $\Lambda_N$:
	\[\lambda\varphi_{N,\lambda}^A(x)+\max_{|v|\leq A}\Big[(-\Delta_N)\varphi_{N,\lambda}^A(x)\cdot v-L(x,v)\Big]=0.\]
	Now recall that due to condition \ref{prel:assumption:super_linear}
	\[L(x,v)\geq g(0).\] Thus, $\varphi_{N,\lambda}^A(x)\geq \lambda^{-1}g(0)$ for each $x\in\Lambda_N$. On the other hand, letting $V_t\equiv 0$, we obtain that $\varphi_{N,\lambda}^A(x)\leq \lambda^{-1}L(z,0)$. Therefore, denoting \[C_0'\triangleq |g(0)|\vee \Big[\sup_{x\in\td}L(x,0)\Big],\] we conclude that 
	\[ |\varphi_{N,\lambda}^A(x)|\leq \lambda^{-1}C_0',\ \ x\in\Lambda_N.\] 
	 Hence, for each $x\in\Lambda_N$ and $i\in \{1,\ldots,d\}$,
	\[\Big|\Delta_{N,i}^+\varphi_{N,\lambda}^A(x)\Big|,\, \Big|\Delta_{N,i}^-\varphi_{N,\lambda}^A(x)\Big|\leq 2N\lambda^{-1} C_0'.\] 
	
	Choosing in condition \ref{prel:assumption:super_linear} $a=2N\lambda^{-1}\sqrt{d}C_0'+1$, we have that, if $x\in\Lambda_N$, $v\in\rd$,
	\[\begin{split}
		(-\Delta_N)\varphi_{N,\lambda}^A(&x)\cdot v-L(x,v)\\&\leq 2N\lambda^{-1}\sqrt{d}C_0'|v|-(2N\lambda^{-1}\sqrt{d}C_0'|v|+1)|v|-g(2N\lambda^{-1}\sqrt{d}C_0'+1)\\&=-|v|-g(2N\lambda^{-1}\sqrt{d}C_0'+1).
	\end{split}\] Hence, 
	\[(-\Delta_N)\varphi_{N,\lambda}^A(x)\cdot v-L(x,v)\rightarrow -\infty\text{ as }|v|\rightarrow\infty\] uniformly w.r.t. $|v|$.
	Thus, there exists a constant $A_{N,\lambda}$ such that, for every $A>A_{N,\lambda}$, 
	\[\max_{|v|\leq A}\Big[(-\Delta_N)\varphi_{N,\lambda}^A(x)\cdot v-L(x,v)\Big]=\sup_{v\in\rd}\Big[(-\Delta_N)\varphi_{N,\lambda}^A(x)\cdot v-L(x,v)\Big].\] Therefore, if $A>A_{N,\lambda}$, the function $\varphi_{N,\lambda}^A$ satisfies
	\[\lambda\varphi_{N,\lambda}^A(x)+\sup_{v\in\rd}\Big[(-\Delta_N)\varphi_{N,\lambda}^A(x)\cdot v-L(x,v)\Big]=0.\] This gives the existence of solution to \eqref{lambda:eq:HJ_Delta}. The fact that $\pi_{\lambda,N}^*$ is optimal, i.e., the fact that equality \eqref{lambda:equality:vaphi_pi_opt} and inequality \eqref{lambda:ineq:J_vaphi} are valid, can be proved using the standard verification arguments. They also give the uniqueness of the solution of \eqref{lambda:eq:HJ_Delta}. 
\end{proof}

\section{Approximation of a solution of the stationary Hamilton-Jacobi equation}\label{sect:rate_lambda}
The aim of this section is to prove the following approximation result.
\begin{theorem}\label{rate_lambda:th:rate}
There exists a constant $C_1$ depending only on the Lagrangian $L$ such that if	 $u_\lambda$, $\varphi_{N,\lambda}$ are solutions of \eqref{intro:eq:lambda_HJ} and \eqref{lambda:eq:HJ_Delta} respectively, while $x\in\Lambda_N$, one has that
	\[|u_\lambda(x)-\varphi_{N,\lambda}(x)|\leq C_1\lambda^{-3/2}N^{-1/2}.\]
\end{theorem}

The proof of this statement relies on several auxiliary statements.
\begin{lemma}\label{rate_lambda:lm:HJ_derivative} There exist constants $c_0$ and $c_1$ depending only on $L$ such that, if
	$\lambda>0$ and $u_\lambda$ is a viscosity solution of \eqref{intro:eq:lambda_HJ}, then, for every $x,y\in\td$,
	\begin{enumerate}
		\item $\lambda|u_\lambda(x)|\leq c_0$;
		\item $|u_\lambda(x)-u_\lambda(y)|\leq c_1|x-y|$.
	\end{enumerate}
\end{lemma}
\begin{proof}
	To prove the first statement, we use condition \ref{prel:assumption:super_linear}. It gives that
	\[L(x,v)\geq g(0).\]  Since $u_\lambda$ is a value function for the discounting control problem \eqref{intro:criterion:stationary}, \eqref{intro:condition:stationary}, we have that
	\[\lambda u_\lambda(z)\geq \lambda\int_0^{+\infty}e^{-\lambda t}g(0)dt=g(0).\] Simultaneously,
	\[\lambda u_\lambda(z)\leq \lambda\int_0^{+\infty}e^{-\lambda t}\sup_{x\in\td}L(x,0)dt=\sup_{x\in\td}L(x,0).\]  Letting 
	\[c_0\triangleq g(0)\vee\Big[\sup_{x\in\td}L(x,0)\Big],\] we derive the first statement.
	
	To prove the second statement, we argue as in the proof of \cite[Proposition 4.1]{Bardi2008c}. We set \[c_1\triangleq c_0+\sup_{x\in\td,|v|=1}L(x,v)+1,\] fix $x\in\tdr$ and consider a test function  $\phi(y)\triangleq c_1|y-x|$. We wish to prove that the maximum of $u_\lambda(y)-\phi(y)$ attains at $x$. Indeed, in the converse case, let $\hat{y}$ be a point where the maximum of $u_\lambda(y)-\phi(y)$ is achieved. Since we assume that $\hat{y}\neq x$, the function $\phi$ is differentiable at $\hat{y}$ with the derivative equal to $c_1\frac{\hat{y}-x}{|\hat{y}-x|}$. Using the fact that $u_\lambda$ is a solution of \eqref{intro:eq:lambda_HJ}, we have that 
	\[\lambda u_\lambda(\hat{y})+H\Bigg(\hat{y},-c_1\frac{\hat{y}-x}{|\hat{y}-x|}\Bigg)\leq 0.\] Plugging in the right-hand side of the Hamiltonian (see~\eqref{prel:intro:H_lattice}) $v\triangleq -\frac{\hat{y}-x}{|\hat{y}-x|}$, we have that 
	\[c_1-\sup_{x\in\td,|v|=1}L(x,v)\leq c_0.\] This contradicts with the choice of $c_1$. Therefore, for each $y\in\td$,
	\[u_\lambda(y)-c_1|y-x|\leq u_\lambda(x).\] Interchanging the variable, we conclude that the function $u_\lambda$  is Lipschitz continuous with the constant $c_1$ that does not depend on $\lambda$.
\end{proof}

\begin{lemma}\label{rate_lambda:lm:v_opt_boundness} Let $\lambda>0$, $x_*\in\td$ and let $(x(\cdot),v(\cdot))$ be such that $v(\cdot)\in L^1([0,+\infty);\td)$, $x(0)=x_*$, $\frac{d}{dt}x(t)=v(t)$ and
	\[u_\lambda(x_*)=\int_{0}^{+\infty}e^{-\lambda t}L(x(t),v(t))dt.\] Then,
	\[|v(t)|\leq c_2\text{ for a.e. }t\in [0,+\infty).\] Here $c_2$ is a constant determined only by the Lagrangian $L$.
\end{lemma}
\begin{proof}
	Let $\tau$ be a Lebesgue point of the function $v(\cdot)$. For $r>0$, from the dynamic programming  arguments we have that 
	\[e^{-\lambda \tau}u_\lambda(x(\tau))=\int_0^re^{-\lambda (\tau+t)}L(x(\tau+t),v(\tau+t))dt+e^{-\lambda (\tau+r)}u_\lambda(x(\tau+r)).\] Therefore, 
	\[
	\int_{0}^{r}e^{-\lambda t}L(x(\tau+t),v(\tau+t))dt=u_\lambda(x(\tau))-u_\lambda(x(\tau+r))+u_\lambda(x(\tau+r))\big[1-e^{-\lambda r}\big]. \] 
	Using Lemma \ref{rate_lambda:lm:HJ_derivative}, we have that 
	\begin{equation}\label{rate_lambda:ineq:int_0_r}\int_{0}^{r}e^{-\lambda t}L(x(\tau+t),v(\tau+t))dt\leq c_1\int_0^r |v(\tau+t)|dt+c_0 r.\end{equation} Now assume that $r\leq \lambda^{-1}\ln 2$. Using this, inequality~\eqref{rate_lambda:ineq:int_0_r} and choosing in superlinear growth condition~\ref{prel:assumption:super_linear} $a=2c_1+2$, we obtain the following inequality
	\[\begin{split}
		(c_1+&1)\int_0^r|v(\tau+t)|dt+2^{-1}g(2c_1+2)r\\
		&\leq (2c_1+1)\int_0^re^{-\lambda t}|v(\tau+t)|dt+g(2c_1+2)\int_0^re^{-\lambda t}dt\\
		&\leq \int_{0}^{r}e^{-\lambda t}L(x(\tau+t),v(\tau+t))dt\leq c_1\int_0^r |v(\tau+t)|dt+c_0 r. \end{split}\] This gives,
	\[\int_0^r|v(\tau+t)|dt\leq \big[c_0-g(2c_1+2)/2\big]r.\] Since $\tau$ is a Lebesgue point of the function $v(\cdot)$, we conclude that 
	\[|v(\tau)|\leq c_2\triangleq \big[c_0-g(2c_1+2)/2\big].\] This and the fact that almost every points of $[0,+\infty)$ are Lebesgue points of the function $v(\cdot)$ completes the proof.
\end{proof}

\begin{lemma}\label{rate_lambda:lm:H_N_derivaitive} Let $\varphi_{N,\lambda}$ solve stationary Hamilton-Jacobi equation \eqref{lambda:eq:HJ_Delta}. Then, for each $x,y\in\Lambda_N$, $i\in \{1,\ldots,d\}$,
	\begin{itemize}
		\item $\lambda|\varphi_{N,\lambda}(x)|\leq c_0$ for each $x\in \Lambda_N$;
		\item $|\Delta_{N,i}^+\varphi_{N,\lambda}(x)|,\, |\Delta_{N,i}^+\varphi_{N,\lambda}(x)|\leq c_3$;
		\item $|\varphi_{N,\lambda}(x)-\varphi_{N,\lambda}(y)|\leq c_4|x-y|$.
	\end{itemize}
	Here, $c_0$ is the same constant as in Lemma \ref{rate_lambda:lm:HJ_derivative}, $c_3$ and $c_4$ are determined only by the Lagrangian $L$.
\end{lemma}
\begin{proof}
	The proof of the first statement is the same as in Lemma \ref{rate_lambda:lm:HJ_derivative}. 
	
	Now let us proves the second statement. From the first statement, we have that, given $x\in\Lambda_N$, 
	\begin{equation*}\label{rate_lambda:ineq:sup_c_0}
		\sup_{v\in\rd}\Big[(-\Delta_N)\varphi_{N,\lambda}(x)-L(x,v)\Big]\leq c_0.
	\end{equation*} 
	If $\Delta_{N,i}^+\varphi_{N,\lambda}(x)\leq 0$, then, letting $v=e_i$ in the expression of the lattice Hamiltonian~\eqref{prel:equality:Ham}, we obtain that 
	\[\Delta_{N,i}^+\varphi_{N,\lambda}(x)\geq -(c_0+L(x,e_i)).\] Analogously, 
	if $\Delta_{N,i}^-\varphi_{N,\lambda}(x)\geq 0$, we have that
	\[\Delta_{N,i}^-\varphi_{N,\lambda}(x)\leq c_0+L(x,-e_i).\] Combining this, we conclude that, 
	\begin{equation*}\label{rate_lambda:ineq:Delta_plus_minus}
		\sup_{x\in\td,i\in \{1,\ldots,d\}}\Big(\Delta_{N,i}^-\varphi_{N,\lambda}(x)\Big)^+,\ \ 
		\sup_{x\in\td,i\in \{1,\ldots,d\}}\Big(\Delta_{N,i}^+\varphi_{N,\lambda}(x)\Big)^-\leq c_3.
	\end{equation*} Here,
	\[c_3\triangleq c_0+ \sup_{x\in\td,|v|\leq 1}L(x,v).\] From this, using equalities \eqref{markov:equality:Delta_shift}, we arrive at the second statement of the lemma.
	
	The third statement obviously follows from this, the definition of right and left differences (see \eqref{markov:intro:Deltas_i}) and the H\"older inequality.
\end{proof}

\begin{lemma}\label{rate_lambda:lm:policy_bound} There exists a constant $c_5$ determined only by the Lagrangian $L$ such that, if $\pi_{\lambda,N}^*$ satisfies \eqref{lambda:incl:pi_star}, then 
	\[|\pi_{\lambda,N}^*(x)|\leq c_5.\]
\end{lemma}
\begin{proof}
	From the definition of the feedback strategy $\pi_{\lambda,N}^*$ and the first statement of Lemma \ref{rate_lambda:lm:H_N_derivaitive}, we have that 
	\[\sum_{i=1}^d\Big[\Delta_{N,i}^+\varphi_{N,\lambda}(x)(\pi_{\lambda,N}^*(x))^+_i+\Delta_{N,i}^-\varphi_{N,\lambda}(x)(\pi_{\lambda,N}^*(x))^-_i\Big]+L(x,\pi_{\lambda,N}^*(x))\leq c_0.\] Here $(\pi_{\lambda,N}^*(x))_i^+$ (respectively, $(\pi_{\lambda,N}^*(x))_i^-$) stands for the positive (respectively, negative) part of the $i$-th coordinate of the vector $\pi_{\lambda,N}^*(x)$. Using  condition \ref{prel:assumption:super_linear} with $a=\sqrt{d}c_3+1$, the second statement of Lemma \ref{rate_lambda:lm:H_N_derivaitive} and the H\"older's inequality, we obtain the following estimate:
	\[-\sqrt{d}c_3|\pi_{\lambda,N}^*(x)|+(\sqrt{d}c_3+1)|\pi_{\lambda,N}^*(x)|\leq c_0-g\big(\sqrt{d}c_3+1\big).\] This implies the conclusion of the lemma with $c_5\triangleq c_0-g\big(\sqrt{d}c_3+1\big)$.
\end{proof}

\begin{proof}[Proof of Theorem \ref{rate_lambda:th:rate}]
	Let $z\in\Lambda_N$. First, let $v(\cdot)$ be such that 
	\begin{equation}\label{rate_lambda:equality:v_u}
		u_\lambda(z)=\int_0^{+\infty}e^{-\lambda t}L(x(t),v(t))dt,
	\end{equation} 
where
	\[x(t)=z+\int_0^tv(s)ds.\] The existence of such open-loop control follows from the Tonelli theorem \cite[Theorem 3.7]{Buttazzo1998}.  Due to Lemma~\ref{rate_lambda:lm:v_opt_boundness}, 
	\[|v(t)|\leq c_2.\]
	
	Let $\pi$ be a feedback strategy defined by the rule:
	\[\pi(t,x)\triangleq v(t).\] Now, we construct a filtered probability space $(\Omega,\mathcal{F},\{\mathcal{F}_t\}_{t\in [0,+\infty)},\mathbb{P})$ and a stochastic process $(X^1,X^2)$ defined on it as in Lemma \ref{markov:lm:distance} for $x_*^1=x^2_*=z$ and $\mathcal{I}=[0,+\infty)$. Notice that $X^1_t=x(t)$ for $t\in [0,+\infty)$. We have that 
	\begin{equation}\label{rate_lambda:ineq:varphi_v}
		\varphi_{N,\lambda}(x_*)\leq \mathbb{E}\int_{0}^{+\infty}e^{-\lambda t}L(X^2_t,\pi(t,X^2_t))dt=\mathbb{E}\int_{0}^{+\infty}e^{-\lambda t}L(X^2_t,v(t))dt.
	\end{equation} Lemma \ref{markov:lm:distance}  gives that
	\[\mathbb{E}|X_t^1-X^2_t|\leq (c_2)^{1/2}d^{1/4}t^{1/2}N^{-1/2}.\] Therefore,
	\[\mathbb{E}|L(X^2_t,v(t))-L(x(t),v(t))|\leq C'_1\mathbb{E}|X_t^1-X^2_t|\leq C_2't^{1/2}N^{-1/2},\] where \[C_1'\triangleq K( c_2),\ \ C_2'=C_1'(c_2)^{1/2}d^{1/4},\] while the function $K$ is defined by \eqref{prel:intro:K_c}.
	 This, equality \eqref{rate_lambda:equality:v_u} and inequality \eqref{rate_lambda:ineq:varphi_v} imply
	\begin{equation}\label{rate_lambda:ineq:u_phi_lambda_1}
		\varphi_{N,\lambda}(z)-u_\lambda(z)\leq C_2'N^{-1/2}\int_0^{+\infty}t^{1/2}e^{-\lambda t}dt=C_2'N^{-1/2}\lambda^{-3/2}\int_0^{+\infty}\alpha^{1/2}e^{-\alpha}d\alpha.
	\end{equation}
	
	To prove the opposite inequality, recall that by Proposition \ref{lambda:prop:existence}, there exists an optimal feedback strategy $\pi_{\lambda,N}^*$ that is stationary. By Lemma \ref{rate_lambda:lm:policy_bound}, 
	\[|\pi_{\lambda,N}^*(x)|\leq c_5.\] Furthermore,
	we construct a filtered probability space $(\Omega,\mathcal{F},\{\mathcal{F}_t\}_{t\in [0,+\infty)},\mathbb{P})$ and a process $(X^1,X^2)$ as in Lemma \ref{markov:lm:distance} for $\pi=\pi_{\lambda,N}^*$, $\mathcal{I}=[0,+\infty)$ and $x_*^1=x^2_*=z$. Denoting
	\[V_t\triangleq \pi_{\lambda,N}^*(X^2_t),\] we conclude that
	\[\frac{d}{dt}X_t^1=V_t,\ \ \mathbb{P}\text{-a.s.}\] Therefore,
	\begin{equation}\label{rate_lambda:ineq:u_L}
		u_\lambda(z)\leq \mathbb{E}\int_0^{+\infty}e^{-\lambda t}L(X^1_t,V_t)dt=\mathbb{E}\int_0^{+\infty}e^{-\lambda t}L(X^1_t,\pi_{\lambda,N}^*(X^2_t))dt.
	\end{equation} 
Additionally, the choice of $\pi^*_{\lambda,N}$ gives that 
	\begin{equation}\label{rate_lambda:equality:varphi_L}
		\varphi_{N,\lambda}(z)=\mathbb{E}\int_0^{+\infty}e^{-\lambda t}L(X^2_t,\pi_{\lambda,N}^*(X^2_t))dt.
	\end{equation} 
	Taking into account the boundness of $\pi_{\lambda,N}^*$, Lemma \ref{markov:lm:distance} and assumption \ref{prel:assumption:c}, we have that 
	\[\mathbb{E}|L(X^1_t,\pi_{\lambda,N}^*(X^2_t))-L(X^2_t,\pi_{\lambda,N}^*(X^2_t))|\leq C_1'\mathbb{E}|X^1_t-X^2_t|\leq C_3't^{1/2}N^{1/2}.\] Here, $C_1'$ is as above, while $C_3'\triangleq C_1'(c_5)^{1/2}d^{1/4}$.
	This, \eqref{rate_lambda:ineq:u_L} and \eqref{rate_lambda:equality:varphi_L} give that 
	\[u_\lambda(z)-\varphi_{N,\lambda}(z)\leq C_3'N^{-1/2}\int_0^{+\infty}t^{1/2}e^{-\lambda t}dt=C_3'N^{-1/2}\lambda^{-3/2}\int_0^{+\infty}\alpha^{1/2}e^{-\alpha}d\alpha. \] The conclusion of the theorem with 
	\[C_1\triangleq (C_2'\vee C_3')\cdot \int_0^{+\infty}\alpha^{1/2}e^{-\alpha}d\alpha \] follows from this and \eqref{rate_lambda:ineq:u_phi_lambda_1}.
\end{proof}

\section{Weak KAM theory on the lattice}\label{sect:weak_KAM}
In this section, we consider the weak KAM theory problem for the controlled Markov chain: 
find a function $\varphi_N :\Lambda_N\rightarrow\mathbb{R}$ and  a constant $\overline{\mathcal{H}}_N$ such that, for each $T>0$ and $z\in\Lambda_N$,
\[\begin{split}
	\varphi_N(z)=\min\Bigg\{\mathbb{E}\bigg[\int_0^TL(X_t&,V_t)dt+\varphi_N(X_T)\bigg]:  \\ 
	(\Omega&,\mathcal{F},\{\mathcal{F}_t\}_{t\in [0,T]},\mathbb{P},X,V)\in\operatorname{MCP}_N([0,T],z)
	\Bigg\}+\overline{\mathcal{H}}_NT.\end{split}\] The dynamic programming arguments give that a pair $(\varphi_N,\overline{\mathcal{H}}_N)$ is a solution of this problem if and only if they satisfies the following weak KAM equation on the lattice $\Lambda_N$:
\begin{equation}\label{weak_KAM:eq:weak_KAM_N}
	\mathcal{H}_N(x,(-\Delta_{N})\varphi_N(x))=\overline{\mathcal{H}}_N,\ \ x\in\Lambda_N.
\end{equation}
\begin{theorem}\label{main_res:th:existence_weak_KAM} For each $N$, there exists a solution of \eqref{weak_KAM:eq:weak_KAM_N}. Moreover, the constant~$\overline{\mathcal{H}}_N$ is unique.
\end{theorem}
\begin{proof}
	We use a method  borrowed from~\cite{lions_et_al}. Due to Lemma \ref{rate_lambda:lm:H_N_derivaitive}, one has that on $\Lambda_N$
	\begin{equation}\label{approx_weak_KAM:ineq:varphi_N_lambda}
		|\lambda\varphi_{N,\lambda}(x)|\leq c_0,
	\end{equation} while, for each $x,y\in\Lambda_N$,
	\[|\varphi_{N,\lambda}(x)-\varphi_{N,\lambda}(y)|\leq c_4|x-y|. \] Hereinafter, $\varphi_{N,\lambda}$ is the unique solution of \eqref{lambda:eq:HJ_Delta}.
	
	Furthermore, we fix, $z\in\Lambda_N$ and set
	\[\psi_{N,\lambda}(x)\triangleq \varphi_{N,\lambda}(x)-\varphi_{N,\lambda}(z).\] By construction, we have that, for each $x\in\Lambda_N$, 
	\begin{equation}\label{approx_weak_KAM:ineq:psi_N_lambda}
		|\psi_{N,\lambda}(x)|\leq C_4'\triangleq c_4\sqrt{d}/2.
	\end{equation}
	Since $\varphi_{N,\lambda}$ solves \eqref{lambda:eq:HJ_Delta}, we have that 
	\begin{equation}\label{approx_weak_KAM:equality:phi_psi} 
		\lambda\varphi_{N,\lambda}+ \mathcal{H}_N(x,(-\Delta_{N})\psi_{N,\lambda}(x))=0,\ \ x\in\Lambda_N.
	\end{equation} 
	From \eqref{approx_weak_KAM:ineq:varphi_N_lambda} and \eqref{approx_weak_KAM:ineq:psi_N_lambda}, it follows that there exist a sequence $\{\lambda_k\}_{k=1}^\infty$, a number  $\overline{\mathcal{H}}_N$ and a function $\varphi_N:\Lambda_N\rightarrow \mathbb{R}$ such that 
	\begin{itemize}
		\item $\lambda_k\rightarrow 0$ as $k\rightarrow\infty$;
		\item $\lambda_k\varphi_{N,\lambda_k}\rightarrow -\overline{\mathcal{H}}_N$ as $k \rightarrow\infty$ uniformly on $\Lambda_N$;
		\item for each $x\in\Lambda_N$, $\psi_{N,\lambda_k}(x)\rightarrow \varphi_N(x)$ as $k\rightarrow\infty$.
	\end{itemize} Passing to the limit in~\eqref{approx_weak_KAM:equality:phi_psi}, we obtain the existence of a solution to the weak KAM equation on the lattice $\Lambda_N$~\eqref{weak_KAM:eq:weak_KAM_N}.
	
	Now let us prove the uniqueness of the number $\overline{\mathcal{H}}_N$. Assume that there exist two solutions of~\eqref{weak_KAM:eq:weak_KAM_N}, namely, $(\varphi_N'(\cdot),\overline{\mathcal{H}}_N')$, $(\varphi_N''(\cdot),\overline{\mathcal{H}}_N'')$. 
	
	We choose a feedback strategy $\pi'_{N}$  such that 
	\begin{equation}\label{approx_weak_KAM:intro:pi_N}
		\pi_N'(x)\in\underset{v\in \rd}{\operatorname{Argmax}}\Big[(-\Delta_{N})\varphi'_{N}(x)\cdot v-L(x,v)\Big].
	\end{equation} 
Let $z\in\Lambda_N$ and let  $(\Omega,\mathcal{F},\{\mathcal{F}_{t}\}_{t\in [0,+\infty)},\mathbb{P},X')$ be a motion produced by the strategy $\pi'_N$ and the initial state $z$. For each $T>0$, we have that 
	\[\begin{split}
		\mathbb{E}[\varphi'(&X'(T))-\varphi'(z)]=
		\mathbb{E}\int_0^T\Delta_{N}\varphi_N'(X_t')dt\\&=
		\mathbb{E}\int_0^T\Big[\Delta_{N}\varphi_N'(X_t')+L(X_t',\pi'_N(X_t'))\Big]dt-\mathbb{E}\int_0^TL(X_t',\pi'_N(X_t'))dt.
	\end{split}\] Due to the choice of the strategy $\pi'_N$ (see \eqref{approx_weak_KAM:intro:pi_N}) and the fact that $(\varphi_N',\overline{\mathcal{H}}'_N)$ is a solution of the weak KAM solution, we have that 
	\begin{equation}\label{approx_weak_KAM:equality:E_varphi_prime}
		\mathbb{E}[\varphi'(z)-\varphi'(X'(T))]= T\overline{\mathcal{H}}'_N+\mathbb{E}\int_0^TL(X_t',\pi'_N(X_t'))dt.
	\end{equation} Simultaneously, 
	\[\begin{split}
		\mathbb{E}[\varphi''(z&)-\varphi''(X'(T))]\\&=
		-\mathbb{E}\int_0^T\Big[\Delta_{N}\varphi_N''(X_t')+L(X_t',\pi'_N(X_t'))\Big]dt+\mathbb{E}\int_0^TL(X_t',\pi'_N(X_t'))dt.
	\end{split}\] Since, for each $x\in\Lambda_N$,
	\[\begin{split}
	\Big[(-\Delta_{N})\varphi_N''(x)-L(x,\pi'_N(x))\Big]&\leq \max_{v\in\rd}\Big[(-\Delta_{N})\varphi_N''(x)-L(x,v)\Big]\\&=\mathcal{H}_N(x,(-\Delta_{N})\varphi_N''(x))=\overline{\mathcal{H}}_N'',\end{split}\] we have that 
	\[\mathbb{E}[\varphi''(z)-\varphi''(X'(T))]\leq T\overline{\mathcal{H}}_N''+\mathbb{E}\int_0^TL(X_t',\pi'_N(X_t'))dt.\] This and \eqref{approx_weak_KAM:equality:E_varphi_prime} yield that 
	\[
	T\overline{\mathcal{H}}_N''-\mathbb{E}[\varphi''(z)-\varphi''(X'(T))]\geq T\overline{\mathcal{H}}_N'-\mathbb{E}[\varphi'(x_0)-\varphi'(X'(T))].
	\] Since that functions $\varphi_N'$ and $\varphi_N''$ are bounded on $\Lambda_N$, dividing both parts of this inequality by $T$ and passing to the limit when $T\rightarrow\infty$, we conclude that 
	\[\overline{\mathcal{H}}_N''\geq \overline{\mathcal{H}}_N'.\] The opposite inequality in proved in the same way. Thus, we conclude that $\overline{\mathcal{H}}_N'=\overline{\mathcal{H}}_N''$.
\end{proof}

\section{Limit  weak KAM equations on  lattices}\label{sect:approx_weak_KAM}
\subsection{Limit of effective Hamiltonians}
In this section, we prove the convergence result for the sequence $\{\overline{\mathcal{H}}_N\}_{N=1}^\infty$.
\begin{theorem}\label{main_res:th:approx_weak_KAM} Let $(u,\bar{H})$ satisfy weak KAM equation for the continuous phase space \eqref{intro:eq:weak_KAM}, $N$ be a natural number, and let $(\varphi_N,\overline{\mathcal{H}}_N)$ solve  weak KAM equation on the lattice $\Lambda_N$~\eqref{weak_KAM:eq:weak_KAM_N}. Then, 
	\[|\bar{H}-\overline{\mathcal{H}}_N|\leq C_2N^{-1/2},\] where $C_2$ is the constant determined only by $L$. 
\end{theorem}
The proof of this theorem  relies on several auxiliary statements.
\begin{lemma}\label{approx_weak_KAM:lm:u} Let $(u,\bar{H})$ solve \eqref{intro:eq:weak_KAM}. Then,
	\begin{enumerate}
		\item $|\bar{H}|\leq c_0$;
		\item $|u(x)-u(y)|\leq c_1|x-y|$.
	\end{enumerate}
\end{lemma}
\begin{proof}
	Recall that, in~\cite{lions_et_al}, it is shown that 
	\[\bar{H}=-\lim_{k\rightarrow\infty}\lambda_ku_{\lambda_k}(x)\] for some sequence $\{\lambda_k\}_{k=1}^\infty$ converging to zero and every $x\in\td$.  This and the the first statement of Lemma \ref{rate_lambda:lm:HJ_derivative} give that $|\bar{H}|\leq c_0$.
	
	The estimate $|u(x)-u(y)|\leq c_1|x-y|$ is proved in the same way as the second statement of Lemma \ref{rate_lambda:lm:HJ_derivative}.
\end{proof}

\begin{lemma}\label{approx_weak_KAM:lm:v_bound} There exists a constant $c_7$ determined only by the Lagrangian such that, if 
	\begin{itemize}
		\item $(u(\cdot),\bar{H})$ is a solution of weak KAM equation~\eqref{intro:eq:weak_KAM},
		\item $T>0$,
		\item	$v(\cdot)\in L^1([0,T];\rd)$ ,
		\item $x(\cdot)\in C([0,T];\rd)$,
	\end{itemize}
	satisfy \[x(t)=x(0)+\int_{0}^tv(s)ds,\ \ u(x_0)-u(x(T))=\int_{0}^TL(x(t),v(t))dt+\bar{H}T,\] 
	then, for a.e. $r\in [0,T]$,
	\[|v(t)|\leq c_6.\]
\end{lemma}
\begin{proof} We argues as in the proof of Lemma \ref{rate_lambda:lm:v_opt_boundness}.
	Let $\tau$ be a Lebesgue point for the function $v(\cdot)$. From the dynamic programming principle, we have that, for $t\in (0,T-\tau]$ 
	\[u(x(\tau))-u(x(\tau+r))=\int_\tau^{\tau+r}L(x(t),v(t))dt+\bar{H}r.\]
	Using the second statement of the Lemma \ref{approx_weak_KAM:lm:u} and assumption \ref{prel:assumption:super_linear} for $a=c_1+1$, we have that 
	\[c_1\int_\tau^{\tau+r}|v(t)|dt\geq (c_1+1)\int_\tau^{\tau+r}|v(t)|dt+(g(c_1+1)+\bar{H})r.\] Therefore,
	\[\int_\tau^{\tau+r}|v(t)|dt\leq -(g(c_1+1)+\bar{H})r.\] Since $\tau$ is a Lebesgue point for $v(\cdot)$, we obtain the conclusion of the lemma with $c_6=-g(c_1+1)+c_0$.
\end{proof}

\begin{lemma}\label{approx_weak_KAM:lm:H_N_derivaitive} Let $(\varphi_{N},\overline{\mathcal{H}}_N)$ be a solution of \eqref{weak_KAM:eq:weak_KAM_N}. Then, for each $x,y\in\Lambda_N$, $i\in \{1,\ldots,d\}$,
	\begin{itemize}
		\item $\lambda|\varphi_{N,\lambda}(x)|\leq c_0$ for each $x\in \Lambda_N$;
		\item $|\Delta_{N,i}^+\varphi_{N,\lambda}(x)|,\, |\Delta_{N,i}^+\varphi_{N,\lambda}(x)|\leq c_3$;
		\item $|\varphi_{N,\lambda}(x)-\varphi_{N,\lambda}(y)|\leq c_4|x-y|$.
	\end{itemize}
	Here, $c_0$ is the same constant as in Lemma \ref{rate_lambda:lm:HJ_derivative}, whereas $c_3$, $c_4$ are determined in Lemma~\ref{rate_lambda:lm:H_N_derivaitive}.
\end{lemma} The proof of this lemma mimics the proof of Lemma~\ref{rate_lambda:lm:H_N_derivaitive}.

Finally, let us evaluate the norm of a feedback strategy determined by a solution of the weak KAM equation on the lattice $\Lambda_N$.
\begin{lemma}\label{approx_weak_KAM:lm:policy_bound} Let $(\varphi_{N},\overline{\mathcal{H}}_N)$ be a solution of weak KAM equation on the lattice $\Lambda_N$~\eqref{weak_KAM:eq:weak_KAM_N} and let $\pi_{N}^*$ be such that \begin{equation}\label{approx_weak_KAM:intro:pi_N_star}
		\pi_N^*(x)\in\underset{v\in \rd}{\operatorname{Argmax}}\Big[(-\Delta_{N})\varphi_{N}(x)\cdot v-L(x,v)\Big],
	\end{equation} then
	\[|\pi_{N}^*(x)|\leq c_5\text{ for each }x\in\Lambda_N.\]
\end{lemma} The proof  literally follows the proof of Lemma~\ref{rate_lambda:lm:policy_bound}, where we utilize the results of Lemma~\ref{approx_weak_KAM:lm:H_N_derivaitive} instead of Lemma~\ref{rate_lambda:lm:H_N_derivaitive}.

\begin{proof}[Proof of Theorem~\ref{main_res:th:approx_weak_KAM}]
We use the trick apparently first proposed in~\cite{Soga2021}.

	Let $z_\natural\in\Lambda_N$ be such that 
	\begin{equation}\label{approx_weak_KAM:ineq:z_max}
		\varphi_{N}(z_\natural)-u(z_\natural)\leq \varphi_{N}(x)-u(x)\text{ for each }x\in\Lambda_N.\end{equation}
	
	Furthermore, let $\pi_{N}^*$ be chosen by rule~\eqref{approx_weak_KAM:intro:pi_N}.
	 Applying Lemma~\ref{markov:lm:distance} with $\mathcal{I}=[0,1]$, $x_*^1=x_*^2=z_\natural$ and $\pi(t,x)=\pi_N^*(x)$, we obtain a filtered probability space $(\Omega,\mathcal{F},\{\mathcal{F}_t\}_{t\in [0,T]},\mathbb{P})$ and a stochastic process $(X^1,X^2)$ defined on it that satisfies conditions~\eqref{markov:eq:tilde_X_ini},~\eqref{markov:intro:X}. Therefore,
	\begin{equation*}\label{approx_weak_KAM:equality:varphi_L}
		\begin{split}
			\mathbb{E}\Big[\varphi_N&(X^2_1)-\varphi_N(z_\natural)\Big]\\&=
			\mathbb{E}\int_{0}^1\Big[\Delta_{N}\varphi_N(X^2_t)\pi_N^*(X^2(t))+L(X^2_t,\pi_N^*(X^2_t))\Big]dt\\&{}\hspace{148pt}-\mathbb{E}\int_{0}^1L(X^2_t,\pi^*_N(X^2_t))dt\\&=
			-\overline{\mathcal{H}}_N-\mathbb{E}\int_{0}^1L(X^2_t,\pi^*_N(X^2_t))dt.
		\end{split}
	\end{equation*} Analogously, due to the fact that $(u,\bar{H})$ solves \eqref{intro:eq:weak_KAM} and the characterization of this solution as a fixed point of Lax-Oleinik operator, we have that 
	\begin{equation*}\label{approx_weak_KAM:ineq:u_L}
		\mathbb{E}\Big[u(z_\natural)-u(X^1_1)\Big]\leq 
		\bar{H}+\mathbb{E}\int_{0}^1L(X^1_t,\pi^*_N(X^2_t))dt.
	\end{equation*} 
Therefore,
\begin{equation*}
\begin{split}
	\overline{\mathcal{H}}_N-\bar{H}\leq \mathbb{E}\big[(\varphi_{N}(&z_\natural)-u(z_\natural))-(\varphi_{N}(X_1^2)-u(X_1^1))\big]\\&-\mathbb{E}\int_0^1 \big[L(X^2_t,\pi_{N}^*(X_t^2))-L(X^1_t,\pi_{N}^*(X_t^2))\big]dt\\ =
	\mathbb{E}\big[(\varphi_{N}(&z_\natural)-u(z_\natural))-(\varphi_{N}(X_1^2)-u(X_1^2))\big]+\mathbb{E}(u(X^1_1)-u(X^2_1))\\&-\mathbb{E}\int_0^1 \big[L(X^2_t,\pi_{N}^*(X_t^2))-L(X^1_t,\pi_{N}^*(X_t^2))\big]dt.
\end{split}
\end{equation*} Due to the choice of $z_\natural$ (see~\eqref{approx_weak_KAM:ineq:z_max}), we have that 
\begin{equation*}\label{approx_weak_KAM:ineq:Hs}
	\overline{\mathcal{H}}_N-\bar{H}\leq \mathbb{E}(u(X^1_1)-u(X^2_1))-\mathbb{E}\int_0^1 \big[L(X^2_t,\pi_{N}^*(X_t^2))-L(X^1_t,\pi_{N}^*(X_t^2))\big]dt.
\end{equation*}

Recall that $|\pi_N^*(x)|\leq c_5$ (this is due to Lemma~\ref{approx_weak_KAM:lm:policy_bound}). Thus, 
Lemmas~\ref{markov:lm:distance},~\ref{approx_weak_KAM:lm:u} and the Jensen's inequality give that 
\[\mathbb{E}(u(X^1_1)-u(X^2_1))\leq C_5'N^{-1/2},\] where $C_5'\triangleq c_1d^{1/4}(c_5)^{1/2}$.

Analogously, Lemma~\ref{markov:lm:distance}, assumption~\ref{prel:assumption:c} and the Jensen's inequality yield that 
	\[\mathbb{E}|L(X^1_t,\pi^*_N(X^2_t))-L(X^2_t,\pi^*_N(X^2_t))|\leq K(c_5)\cdot\mathbb{E}|X^1_t-X^2_t|\leq C_6'N^{-1/2}t^{1/2},\] where the function $K$ is introduced by~\eqref{prel:intro:K_c},  and \[C_6'\triangleq d^{1/4}(c_5)^{1/2}\cdot K(c_5).\]
	This, equality~\eqref{approx_weak_KAM:equality:varphi_L} and inequality~\eqref{approx_weak_KAM:ineq:u_L} imply the  estimate
	\begin{equation}\label{approx_weak_KAM:ineq:H_H_N}
		\overline{\mathcal{H}}_N-\bar{H} \leq  C_7'N^{-1/2},
	\end{equation} for $C_7'\triangleq \frac{3}{2}C_6'+C_5'$. 
	
	To derive the opposite inequality, we first choose $z^\natural$ such that 
		\begin{equation}\label{approx_weak_KAM:ineq:z_min}
		\varphi_{N}(z^\natural)-u(z^\natural)\geq \varphi_{N}(x)-u(x)\text{ for each }x\in\Lambda_N.
	\end{equation}
	
	There exist  functions $v(\cdot)\in L^1([0,1];\rd)$ and $x(\cdot)\in C([0,T];\rd)$ such that 
	\[x(t)=z^\natural+\int_{0}^t v(s)ds\] and 
	\begin{equation}\label{approx_weak_KAM:equality:u_L_calibrated}
		u(z^\natural)-u(x(1))=\int_{0}^1 L(x(t),v(t))+\bar{H}.
	\end{equation} 
	Lemma \ref{approx_weak_KAM:lm:v_bound} says that $|v(t)|\leq c_6$.
	
	Letting $\pi(t,x)\triangleq v(t)$,  we construct a filtered probability space $(\Omega,\mathcal{F},\{\mathcal{F}_t\}_{t\in [0,T]},\mathbb{P})$ and a stochastic process $(X^1,X^2)$ defined on it that satisfies conditions~\eqref{markov:eq:tilde_X_ini},~\eqref{markov:intro:X} for the initial points $x_*^1=x_*^2=z^\natural$.  Notice that $X^1$ now is deterministic with
	\[\frac{d}{dt}X_t^1=v(t),\ \ \mathbb{P}\text{-a.s.}\] Due to the fact that $(\varphi_N,\overline{\mathcal{H}}_N)$ is a solution of~\eqref{weak_KAM:eq:weak_KAM_N}, we have that 
	\begin{equation*}\label{approx_weak_KAM:ineq:varphi_H_N}
		\mathbb{E}\Big[\varphi_N(z^\natural)-\varphi_N(X^2_1)\Big]\leq \overline{\mathcal{H}}_N+\mathbb{E}\int_{0}^1L(X^2_t,v(t))dt.
	\end{equation*}
From this and~\eqref{approx_weak_KAM:equality:u_L_calibrated}, we have that 
\begin{equation}\label{approx_weak_KAM:ineq:H_H_N_below}
	\begin{split}
		\overline{\mathcal{H}}_N-\bar{H}\geq \mathbb{E}\big[(\varphi_{N}(&z^\natural)-u(z^\natural))-(\varphi_{N}(X_1^2)-u(X_1^1))\big]\\&-\mathbb{E}\int_0^1 \big[L(X^2_t,\pi_{N}^*(X_t^2))-L(X^1_t,\pi_{N}^*(X_t^2))\big]dt\\ =
		\mathbb{E}\big[(\varphi_{N}(&z^\natural)-u(z^\natural))-(\varphi_{N}(X_1^2)-u(X_1^2))\big]+\mathbb{E}(u(X^1_1)-u(X^2_1))\\&-\mathbb{E}\int_0^1 \big[L(X^2_t,\pi_{N}^*(X_t^2))-L(X^1_t,\pi_{N}^*(X_t^2))\big]dt\\ \geq \mathbb{E}(u(X^1_1&)-u(X^2_1))-\mathbb{E}\int_0^1 \big[L(X^2_t,\pi_{N}^*(X_t^2))-L(X^1_t,\pi_{N}^*(X_t^2))\big]dt.
	\end{split}
\end{equation} In the latter inequality we used the choice of $z^\natural$ (see \eqref{approx_weak_KAM:ineq:z_min}).
Since $|v(t)|\leq c_6$, Lemmas~\ref{markov:lm:distance},~\ref{approx_weak_KAM:lm:H_N_derivaitive} and the Jensen's inequality yield the estimate
\[\mathbb{E}(u(X^1_1)-u(X^2_1))\geq -C_8'N^{-1/2},\] where $C_8'\triangleq c_1d^{1/4}(c_6)^{1/2}$.
Using the same arguments and  assumption~\ref{prel:assumption:c}, we conclude that
	\[\mathbb{E}|L(X^1_t,v(t))-L(X^2_t,v(t))|\leq K(c_6)\cdot\mathbb{E}|X^1_t-X^2_t|\leq C_9'N^{-1/2}t^{1/2},\] where \[C_9'\triangleq K(c_6) d^{1/4}(c_6)^{1/2}.\] 
	Evaluating the right-hand side of~\eqref{approx_weak_KAM:ineq:H_H_N_below} by means of  these inequalities, we conclude that  
	\[\overline{\mathcal{H}}_N-\bar{H}\geq - C_{10}'N^{-1/2}C_8'N^{-1/2},\] where $C_{10'}\triangleq C_8'+\frac{3}{2}C_9'$.  This and estimate~\eqref{approx_weak_KAM:ineq:H_H_N} provide the conclusion of the theorem with $C_2\triangleq C_7'\vee C_{10}'$.
	
\end{proof}

\subsection{Limits of the functions $\varphi_{N}$} 
Let $(\varphi_{N},\overline{\mathcal{H}}_N)$ be a solution of~\eqref{weak_KAM:eq:weak_KAM_N}.
Lemma \ref{approx_weak_KAM:lm:H_N_derivaitive} implies that the function $\varphi_{N}$ is Lipschitz continuous on $\Lambda_N$ with the Lipschitz constant equal to $c_4$. Below, we denote by $\hat{\varphi}_N$ an extension of the function $\varphi_{N}$ onto the whole torus $\td$ that is Lipschitz continuous with the same constant, i.e., $\hat{\varphi}_N:\td\rightarrow\mathbb{R}$ is such that 
\begin{itemize}
	\item $\hat{\varphi}_N(x)=\varphi_{N}(x)$ whenever $x\in\Lambda_N$;
	\item $\hat{\varphi}_n$ is $c_4$-Lipschitz continuous function.
\end{itemize} The existence of such function can be shown using, for example, the McShane's methods. In this case, we let
\[\hat{\varphi}_N(x)\triangleq \min\big\{\varphi_{N}(y)+c_4|x-y|:\, y\in\Lambda_N\big\}.\]
Notice that the functional part of a solution of the weak KAM equation is defined up to an additive constant. Thus, it is reasonable to fix a value at some point. For definiteness, we choose  this point equal to $\mathbf{0}$ that is the equivalent class corresponding to points with integer coordinates. Thus, without loss of generality, we assume that 
\[\hat{\varphi}_N(\mathbf{0})=0.\]
\begin{theorem}\label{approx_weak_KAM:th:function}
	The sequence of functions $\{\hat{\varphi}_N\}_{N=1}^\infty$ is precompact in $C(\td)$. If a function $u:\td\rightarrow\mathbb{R}$ is  its accumulation point, then $(u,\bar{H})$ is a solution of the weak KAM equation on the torus~\eqref{intro:eq:weak_KAM}.
\end{theorem}
\begin{proof}
	The fact that  the sequence $\{\hat{\varphi}_N\}_{N=1}^\infty$ is precompact directly follows from its definition.
	
	To prove the second part, we assume that there exist a sequence $\{N_l\}_{l=1}^\infty$ and a Lipschitz continuous function $u$ such that 
	\[\|\hat{\varphi}_{N_l}-u\|\rightarrow 0\text{ as }l\rightarrow\infty.\]
	
	Further, let 
	\begin{itemize}
	 \item $x\in \td$; 
	 \item $\{z_{l}\}_{l=1}^\infty\subset \td$ such that $z_l\in\Lambda_{N_l}$ and $z_l\rightarrow x$ as $l\rightarrow\infty$;  \item $v\in \rd$.
	\end{itemize}
	For each natural $l$, we construct a 6-tuple $(\Omega^l,\mathcal{F}^l,\{\mathcal{F}^l\}_{t\in [0,+\infty)},\mathbb{P}^l,X^{l,1},X^{l,2})$ satisfying the conditions of Lemma~\ref{markov:lm:distance} for $N=N_l$, the strategy $\pi(t,x)\equiv v$ and $x_*^1=x_*^2=z_l$. Thus,  
	\begin{equation}\label{approx_weak_KAM:ineq:X_1_X_2}
		\mathbb{E}^l|X^{l,1}_t-X^{l,2}_t|^2\leq \sqrt{d}|v|tN_l^{-1}.\end{equation} Here $\mathbb{E}^l$ is an expectation corresponding to the probability $\mathbb{P}^l$.
	
	Notice that the process $X^{1,l}$ is deterministic and satisfies
	\[X^{l,1}_t=z_l+vt,\ \ \mathbb{P}\text{-a.s.}\]
	Let $T$ be a positive number. We have that 
	\[\varphi_{N_l}(z_l)-\mathbb{E}^l\varphi_{N_l}(X^{l,2}_T)
	\leq \mathbb{E}^l\int_0^TL(X^{l,2}_t,v)dt+\overline{\mathcal{H}}_{N_l}T.
	\] From~\eqref{approx_weak_KAM:ineq:X_1_X_2}, Lemma~\ref{approx_weak_KAM:lm:H_N_derivaitive} and the definition of the function $K$ (see~\eqref{prel:intro:K_c}), we conclude that 
	\[\begin{split}
	\varphi_{N_l}(&x)-\varphi_{N_l}(x+vT)-L(z_l,v)T
	\\&\leq \overline{\mathcal{H}}_{N_l}T+c_4(d^{1/4}|v|^{1/2}T^{1/2}N_l^{-1/2}+2|z_l-x|)\\&{}\hspace{70pt}+K(|v|)d^{1/4}|v|^{1/2}T^{3/2}N_l^{-1/2}+K(|v|)|v|T^{3/2}.
\end{split}\] Passing to the limit when $l\rightarrow \infty$, we obtain that 
\[u(x)-u(x+vT)-L(x,v)T\leq \bar HT+K(|v|)|v|T^{3/2}.\]

Now let $\psi$ be a smooth function such that, for some $r>0$, the mapping $\mathbb{B}_r(x)\ni y\mapsto u(y)-\psi(y)$ attains the maximum at $x$. If $|v|T<r$, we have that 
 \[\psi(x)-\psi(x+vT)-L(x,v)T\leq \bar{H}T+K(|v|)|v|T^{3/2}.\] Dividing both parts by $T$ and passing to the limit when $T\rightarrow 0$, we obtain that 
 \[-\nabla\psi(x) v-L(x,v)\leq \bar{H}.\] Since the choice of $v$ is arbitrarily,  we conclude that
 \begin{equation}\label{approx_weak_KAM:ineq:viscosity_sub}
 	H(x,-\nabla\psi(x))\leq \bar{H}.
 \end{equation}	

This provides  the first part of the definition of the viscosity solution.
Let us prove the second part.

As above, given $x\in\td$, we consider a sequence $\{z_l\}_{l=1}^\infty\subset\td$ converging to $x$ such that $z_l\in\Lambda_{N_l}$. For each natural $l$, let $\pi_{N_l}^*$ be such that 
\[\pi_{N_l}^* \in\underset{v\in \rd}{\operatorname{Argmax}}\Big[(-\Delta_{N_l})\varphi_{N_l}(x)\cdot v-L(x,v)\Big].\]
 Lemma~\ref{approx_weak_KAM:lm:policy_bound} says that 
 \[|\pi_{N_l}^*|\leq c_5.\] There exists a  6-tuple $(\Omega^l,\mathcal{F}^l,\{\mathcal{F}^l\}_{t\in [0,+\infty},\mathbb{P}^l,X^{l,1},X^{l,2})$ satisfying conditions of Lemma~\ref{markov:lm:distance} for $x_*^1=x_2^*=z_l$ and the stationary strategy $\pi_{N_l}^*$. Arguing as in the proof of the first viscosity inequality, we conclude that 
 \[\mathbb{E}^l|X^1_t-X^2_t|^2\leq \sqrt{d}c_5tN_l. \] Here, as above, $\mathbb{E}^l$ stands for the expectation corresponding to the probability $\mathbb{P}^l$. Due to the choice of the strategy $\pi_{N_l}^*$, we have that 
 \begin{equation}\label{approx_weak_KAM:equality:varphi_N_l_L}
 	\varphi_{N_l}(z_l)- \mathbb{E}^l\varphi_{N_l}(X^{l,2}_T)= \mathbb{E}^l\int_{0}^T L(X^{l,2}_T,\pi_{N_l}^*(X^{l,2}_t))dt+\overline{\mathcal{H}}_{N_l}T.
 \end{equation} 
Put
 $V^l(t)\triangleq \pi_{N_l}^*(X^{2,l}_t)$.  The construction of the processes $X^{1,l}$, $X^{2,l}$ gives that 
 \[X_t^{1,l}=z_l+\int_0^t V^l(s)ds,\ \ \mathbb{P}\text{-a.s.}\]
  Thanks to Lemma~\ref{approx_weak_KAM:lm:u}, the Lipschitz continuity of the function $\hat{\varphi}_{N_l}$, the definition of the function $K$ (see \eqref{prel:intro:K_c}) and equality~\eqref{approx_weak_KAM:equality:varphi_N_l_L}, we have that 
 \[\begin{split}
 u(x)-\mathbb{E}^lu\Bigg(x+\int_{0}^TV^l(t)dt\Bigg)\geq &\mathbb{E}\int_{0}^T L(x,V^l(t))dt+\overline{\mathcal{H}}_{N_l}T \\- 2\|u-\varphi_{N_l}\|&-c_4(2|z_l-x|+d^{1/4}c_5T^{1/2}N_{l}^{-1/2})\\&-K(c_5)(|z_l-x|+c_5d^{1/4}T^{3/2}N^{-1/2}_l+c_5T^{3/2}).
 \end{split}\] Now let $\psi$ be a smooth function defined in $\mathbb{B}_r(x)$  for some $r>0$ such that the mapping $\mathbb{B}_r(x)\ni y\mapsto u(y)-\psi(y)$ attains the minimum at $x$. For $T$ such that $c_5T<r$, we have that 
\[\psi(x)-\mathbb{E}^l\psi\Bigg(x+\int_{0}^TV^l(t)dt\Bigg)\geq u(x)- \mathbb{E}^lu\Bigg(x+\int_{0}^TV^l(t)dt\Bigg).\] Additionally, since $|V^l(t)|\leq c_5$, $\mathbb{P}$-a.s., the following inequality holds true:
\[H(x,-\nabla\psi(x))T\geq \psi(x)-\mathbb{E}^l\psi\Bigg(x+\int_{0}^TV^l(t)dt\Bigg)-\mathbb{E}^l\int_{0}^T L(x,V^l(t))dt+o(T),\] where $o(T)/T\rightarrow 0$ as $T\rightarrow 0$ uniformly w.r.t. $l$. Here we used the definition of the Hamiltonian $H$. Thus,
 \[\begin{split}
	H(x,-\nabla\psi(x))T\geq \overline{\mathcal{H}}_{N_l}T \\- 2\|u-\varphi_{N_l}\|&-c_1(2|z_l-x|+d^{1/4}c_5T^{1/2}N_{l}^{-1/2})\\&-K(c_5)(|z_l-x|+c_5d^{1/4}T^{3/2}N^{-1/2}_l+c_5T^{3/2})+o(T).
\end{split}\] Passing to the limit  when $l\rightarrow\infty$, we have that 
\[	H(x,-\nabla\psi(x))T\geq \overline{\mathcal{H}}_{N_l}T+K(c_5)c_5T^{3/2}+o(T).\] Dividing this inequality by $T$ and passing to the limit as $T\rightarrow 0$, we have that 
\[	H(x,-\nabla\psi(x))\geq \overline{\mathcal{H}}_{N_l}\] for each smooth  function $\psi$ such that the mapping $\mathbb{B}_r(x)\ni y\mapsto u(y)-\psi(y)$ attains the minimum at $x$. This together with \eqref{approx_weak_KAM:ineq:viscosity_sub} gives the fact that $u$ is a viscosity solution of~\eqref{intro:eq:weak_KAM}.
\end{proof}

\section{Mather measures on $\Lambda_N$ and their limit behavior}\label{sect:Mather}

As in the continuous space case, the constant $\overline{\mathcal{H}}_N$ can be characterized using a Mather measure. To introduce this concept for the lattice system, we first denote by $\mathcal{M}_N$ the set of probabilities on $\LNr$ with finite integral of the function $\mathscr{w}$ defined by~\eqref{prel:intro:weight}, i.e,
\[\mathcal{M}_N\triangleq \Bigg\{\nu\in\mathcal{P}(\LNr):\int_{\LNr}\mathscr{w}(v)\nu(d(x,v))<\infty\Bigg\}.\] Here the function $\mathscr{w}$ is defined by~\eqref{prel:intro:weight}.
Obviously, $\mathcal{M}_N\subset \mathcal{M}$.
\begin{definition}\label{mather:def:holonomic} A measure $\mu\in \mathcal{P}(\Lambda_N\times\rd)$ is called a holonomic for the lattice $\Lambda_N$ if, for each $\phi:\Lambda_N\rightarrow\mathbb{R}$, the following equality holds true:
	\[\int_{\Lambda_N\times \rd}\big(\Delta_{N}\phi(x)\cdot v\big)\mu(d(x,v))=0.\] 
\end{definition}
\begin{definition}\label{mather:def:mather} A measure $\mu_N\in\mathcal{M}_N$ is called a Mather measure for the lattice $\Lambda_N$ provided that 
	\begin{itemize}
		\item  $\mu_N$ is holonomic for the lattice $\Lambda_N$;
		\item 
		\[\begin{split}\int_{\LNr}L(x&,v)\mu_N(d(x,v))\\&=\min\Bigg\{\int_{\LNr}L(x,v)\nu_N(d(x,v)):\ \ \nu\in\mathcal{M}_N,\, \nu\text{ is holonomic} \Bigg\}.\end{split}\]
	\end{itemize}
\end{definition}

Below, we denote by $\overline{\mathbb{B}}_c$ the closed ball of the radius $c$ in $\rd$ centered in the origin.

\begin{theorem}\label{mather:th:mather_ex}
	There exists a constant $C_3$ such that, for each $N$, one can find a Mather measure for the lattice $\Lambda_N$  concentrated on $\LNB{C_3}$. 
\end{theorem}
\begin{proof}
	Let $(\varphi_N,\overline{\mathcal{H}}_N)$ be a solution of the weak KAM equation on the lattice $\Lambda_N$. Let $\pi_N^*$ be defined by~\eqref{approx_weak_KAM:intro:pi_N_star}. By Lemma~\ref{approx_weak_KAM:lm:policy_bound},  $|\pi_{N}^*(x)|\leq c_5$ on $\Lambda_N$. The Kolmogorov matrix $\mathcal{Q}^N[\pi_N^*]$ determines a Markov chain on the lattice $\Lambda_N$. It is well known (see, for instance, \cite[Theorem 5.4.6]{Strook_markov}) that this Markov chain has at least one stationary distribution, i.e., a sequence $\hat{m}_N=\{\hat{m}_{N,x}\}_{x\in\Lambda_N}$ such that, for each $y\in \Lambda_N$,
	\begin{equation}\label{mather:equality:Q}
		\sum_{x\in\Lambda_N}\hat{m}_{N,x}\mathcal{Q}_{x,y}^N[\pi_{N}]=0.\end{equation}
	
	Now let $\phi:\Lambda_N\rightarrow\mathbb{R}$. Taking into account the definition of the matrix $\mathcal{Q}^N$, from~\eqref{markov:eq:Q_Delta} and~\eqref{mather:equality:Q}, we obtain that 
	\begin{equation}\label{mather:equality:m_Q_phi}
		\begin{split}
			0&=\sum_{y\in\Lambda_N}\Bigg[\sum_{x\in\Lambda_N}\hat{m}_{N,x} \mathcal{Q}_{x,y}^N[\pi_N^*]\Bigg]\phi(y)= \sum_{x\in\Lambda_N}\hat{m}_{N,x}\Bigg[\sum_{y\in\Lambda_N} \mathcal{Q}_{x,y}^N[\pi_N^*]\phi(y)\Bigg]\\&=
			\sum_{x\in\Lambda_N}\hat{m}_{N,x}\sum_{i=1}^d\Big(\Delta_{N,i}^+\phi(x) (\pi_N^*(x))_i^++ \Delta_{N,i}^-\phi(x) (\pi_N^*(x))_i^-\Big)\\&=\sum_{x\in\Lambda_N}\hat{m}_{N,x} \Big(\Delta_{N}\phi(x)\cdot (\pi_N^*(x))\Big).
		\end{split}
	\end{equation} Set
	\begin{equation}\label{mather:intro:hat_mu}\hat{\mu}_N\triangleq \sum_{x\in\Lambda_N}\hat{m}_{N,x}\delta_{(x,\pi_{N}^*(x))}.\end{equation} Here, $\delta_w$ stands for the Dirac measure concentrated at $w$. Due to~\eqref{mather:equality:m_Q_phi}, we have that 
	\[\int_{\LNr}\Big(\Delta_{N}\phi(x)\cdot v\Big)\hat{\mu}_N(d(x,v))=0,\] i.e., $\hat{\mu}_N$ is holonomic. Furthermore, $\hat{\mu}_N$ is concentrated on $\LNB{C_3}$ for the constant $C_3\triangleq c_5$ that does not depend on a number $N$. 
	
	Now let us prove that 
	\begin{equation}\label{mather:equality:hat_mu_H_N}
		\int_{\LNr}L(x,v)\hat{\mu}_N(d(x,v))=-\overline{\mathcal{H}}_N.
	\end{equation}
	Recall that, due to the choice of the strategy $\pi_N^*$, we have that, for each $x\in\Lambda_N$,
	\[\Delta_{N}\varphi_{N}(x)\cdot\pi_N^*(x)+L(x,\pi_{N}^*(x))=-\overline{\mathcal{H}}_N.\] Multiplying each equality on $\hat{m}_x$ and summing up, we obtain that
	\[\sum_{x\in\Lambda_N}\Bigg[\hat{m}_{x}\big(\Delta_{N}\varphi_{N}(x)\cdot\pi_N^*(x)+L(x,\pi_{N}^*(x))\big)\Bigg]=-\overline{\mathcal{H}}_N.\] Equality~\eqref{mather:equality:m_Q_phi} and the definition of the measure $\hat{\mu}_N$ (see~\eqref{mather:intro:hat_mu}) yield that~\eqref{mather:equality:hat_mu_H_N} holds true.

	To complete the proof it suffices to show that, for each probability $\nu\in\mathcal{M}_N$ that is holonomic, one has that 
	\begin{equation}\label{mather:equality:nu_H_N}
		\int_{\LNr}L(x,v)\nu(d(x,v))\geq -\overline{\mathcal{H}}_N.
	\end{equation} Indeed, given a holonomic measure $\nu$, let
	$m_{x}\triangleq \nu(\{x\times \rd\})$. Further, denote by $\nu_{x}$ a probability on $\rd$ such that, for each Borel set $\Upsilon\subset \rd$
	\[\nu_{x}(\Upsilon)=\nu(\{x\}\times\Upsilon).\] Additionally,
	we define a feedback strategy $\bar{\nu}$ by the rule: if $x\in\Lambda_N$,
	\begin{equation}\label{mather:intro:nu_bar}
		\bar{\nu}(x)\triangleq  \int_{\rd}v\nu_{x}(dv).
	\end{equation}
	It is convenient to set, for each $i\in \{1,\ldots,d\}$,
	\[\tilde{\nu}_i^+(x)\triangleq \int_{\rd}v^+_i\nu_{x}(dv),\ \ \tilde{\nu}_i^-(x)\triangleq \int_{\rd}v^-_i\nu_{x}(dv).\]
	Obviously,	\[\bar{\nu}(x)\triangleq \sum_{i=1}^d \Big(\tilde{\nu}_i^+(x)-\tilde{\nu}_i^-(x)\Big)e_i.\] 
	Notice that the numbers $\tilde{\nu}_i^+(x)$ and $\tilde{\nu}_i^-(x)$ are nonnegative. For each  $\phi:\Lambda_N\rightarrow\rd$ and $x\in\Lambda_N$, the following equalities hold true
	\[\begin{split}
		\int_{\rd}\Big(\Delta_{N}\phi(x)\cdot v\Big)\nu_{x}(dv)&=\int_{\rd}\sum_{i=1}^d\Big( \Delta_{N,i}^+\phi(x)\cdot v_i^++\Delta_{N,i}^-\phi(x)\cdot v_i^-\Big)\nu_{x}(dv)\\ &=\sum_{i=1}^d \Big[\Delta_{N,i}^+\phi(x)\tilde{\nu}_i^+(x)+\Delta_{N,i}^-\phi(x)\tilde{\nu}_i^-(x)\Big]=\Delta_{N}\phi(x)\cdot \bar{\nu}(x).\end{split}\] This and the definition of the sequence $m=\{m_{x}\}_{x\in\Lambda_N}$ imply  the following property  for each function $\phi:\Lambda_N\rightarrow\rd$: 
	\begin{equation}\label{mather:equality:holonomic_m}
		0=\int_{\LNr}\Big(\Delta_{N}\phi(x)\cdot v\Big)\nu(d(x,v))= \sum_{x\in\Lambda_N} m_{x}\Big(\Delta_{N}\phi(x)\cdot \bar{\nu}(x)\Big).
	\end{equation}
	
	Now recall that $(\varphi_{N},\overline{\mathcal{H}}_N)$ is a solution of~\eqref{weak_KAM:eq:weak_KAM_N}. Therefore, using the definition of the Hamiltonian $\mathcal{H}_N$, we have that 
	\[
	(-\Delta_{N})\varphi_{N}(x)\cdot\bar{\nu}(x)-L(x,\bar{\nu}(x))\leq \overline{\mathcal{H}}_N.
	\] Taking into account equality~\eqref{mather:equality:holonomic_m}, equality~\eqref{mather:intro:nu_bar} and the convexity of the function $L$ w.r.t. the second variable, we deduce the following estimates:
	\[\begin{split}
		-\overline{\mathcal{H}}_N\leq \sum_{x\in\Lambda_N} m_{x}&\Big[L(x,\bar{\nu}(x))+\Delta_{N}\varphi_{N}(x)\cdot\bar{\nu}(x)\Big]
		\\&=\sum_{x\in\Lambda_N} m_{x}L\Bigg(x,\int_{\rd}v\nu_{x}(dv)\Bigg)\leq 
		\sum_{x\in\Lambda_N} m_{x}\int_{\rd}L(x,v)\nu_{x}(dv)\\&=\int_{\LNr}L(x,v)\nu(d(x,v))\end{split}\] This, in fact, is inequality~\eqref{mather:equality:nu_H_N}. Thus, $\hat{\mu}_N$ is a Mather measure, that is concentrated on $\LNB{C_3}$.
\end{proof}

We complete this section with  the statement providing the limit behavior of Mather measures. Following \cite{Biryuk2010,Sorrentino}, we  put
\[C_L^0(\tdr)\triangleq \Bigg\{\phi\in C(\tdr):\ \ \sup_{(x,v)\in\tdr} \Bigg|\frac{\phi(x,v)}{\mathscr{w}^+(v)}\Bigg|<\infty,\, \lim_{|v|\rightarrow\infty} \frac{\phi(x,v)}{\mathscr{w}(v)}=0\Bigg\},\] where \[\mathscr{w}^+(v)\triangleq \sup_{x\in\td}L(x,v)\vee 1.\] We say that  a sequence $\{\mu_n\}_{n=1}^\infty\subset \mathcal{M}$ converges $(C_L^0)^*$-weakly to $\mu$ iff, for each $\phi\in  C_L^0$,
\[\int_{\tdr}\phi(x,v)\mu_n(d(x,v))\rightarrow \int_{\tdr}\phi(x,v)\mu(d(x,v)).\] Notice that if $\{\mu_n\}_{n=1}^\infty$ converges to $\mu$ $(C_L^0)^*$-weakly, then it converges to the same measure in the narrow sense.

The limiting properties of Mather measures are described as follows.
\begin{theorem}\label{mather:th:mather_limit} Let 
	\begin{itemize}
		\item $\{N_k\}_{k=1}^\infty$ be an increasing sequence of natural numbers;
		\item for each $k$, $\mu_{N_k}$ be a Mather measure for $\Lambda_{N_k}$;
		\item $\{\mu_{N_k}\}_{k=1}^\infty$ $(C_L^0)^*$-weakly converges to a measure $\mu$.
	\end{itemize} Then, $\mu$ is a Mather measure on the torus.
\end{theorem} Before the proof of this theorem, we introduce the following result.
\begin{corollary}\label{mather:corollary:mather} The sequence of Mather measures $\{\hat{\mu}_N\}_{N=1}^\infty$ defined in Theorem~\ref{mather:th:mather_ex} is precompact. Each its accumulation point is a Mather measure for the continuous phase space. In particular, there exists a Mather measure for the flat torus $\td$ supported on the compact set $\LNB{C_3}$.   
\end{corollary}
\begin{proof}[Proof of Theorem \ref{mather:th:mather_limit}]
	Let $\phi\in C^1(\td)$, we have that, if $x\in\Lambda_N$,
	\[|\Delta_{N,i}^+\phi(x)-\nabla\phi(x)|,|\Delta_{N,i}^-\phi(x)-\nabla\phi(x)|\leq\varsigma_{N,i},\] where the number $\varsigma_{N,i}$ is equal to 
	\[\varsigma_{N,i}[\phi]\triangleq \sup_{\varepsilon\in [-h,h]}|\partial_{x_i}\phi(x+\varepsilon e_i)-\partial_{x_i}\phi(x)|.\] Denote
	\[\varsigma_N[\phi]\triangleq \Bigg[\sum_{i=1}^d(\varsigma_{N,i}[\phi])^2\Bigg]^{1/2}.\]
	Notice that $\varsigma_{N}[\phi]\rightarrow 0$ as $N\rightarrow \infty$. Therefore, for each $\phi\in C^1(\td)$,  
	\[\Big|\Delta_{N}\phi(x)\cdot v-\nabla\phi(x)v\Big|\leq \varsigma_{N}[\phi]\cdot|v|.\]
	
	We have that 
	\begin{equation}\label{mather:equality:int_N_k}
		\begin{split}	
			\int_{\LNkr}\nabla\phi( x)v\mu_{N_k}&(d(x,v))\\=\int_{\LNr}&\Delta_{N_k}\phi( x)\cdot v\mu_{N_k}(d(x,v))\\&-\int_{\LNkr}\Big[\Delta_{N_k}\phi( x)\cdot v-\nabla\phi( x)v\Big]\mu_{N_k}(d(x,v)).\end{split}\end{equation}
	
	The estimate of $\Delta_{N_k}\phi(x)\cdot v-\nabla\phi(x)v$ gives that 
	\[\int_{\LNkr}\Big|\Delta_{N_k}\phi( x)\cdot v-\nabla\phi( x)v\Big|\mu_{N_k}(d(x,v))\leq \varsigma_{N_k}[\phi]\int_{\LNkr}|v|\mu_{N_k}(d(x,v)).\]
	
	Since $L(x,v)\geq |v|+g(1)$, we have that the functions $(x,v)\mapsto |v_i|$ and $(x,v)\mapsto |v|$ lies in $C_L^0$. Thus, the integrals
	$\int_{\LNkr}\nabla\phi(x)v\mu_{N_k}(d(x,v))$, $\int_{\LNkr}\Delta_{N_k}\phi(x)v\mu_{N_k}(d(x,v))$ are well-defined.
	Using $(C_L^0)^*$-weak convergence of the sequence $\{\mu_{N_k}\}$, we conclude that the quantities $\int_{\LNkr}|v|\mu_{N_k}(d(x,v))$ are uniformly bounded. Using this, the fact that each measure $\{\mu_{N_k}\}$ is holonomic for the lattice $\Lambda_{N_k}$, the convergence $\varsigma_{N_k}[\phi]\rightarrow 0$ as $k\rightarrow\infty$, we pass to the limit in~\eqref{mather:equality:int_N_k} and obtain that 
	\[\int_{\tdr}\nabla\phi(x)v\mu(d(x,v))=0.\] Thus, $\mu$ is holonomic.
	
	Since
	\[\int_{\LNkr}L(x,v)\mu_{N_k}(d(x,v))=-\overline{\mathcal{H}}_{N_k},\] using the $(C_L^0)^*$-weak convergence of the sequence $\{\mu_{N_k}\}$ to $\mu$ and Theorem~\ref{main_res:th:approx_weak_KAM}, we arrive at the equality
	\[\int_{\tdr}L(x,v)\mu(d(x,v))=-\bar{H}.\] This gives that $\mu$ is a Mather measure for the flat torus.
\end{proof}

\begin{proof}[Proof of Corollary~\ref{mather:corollary:mather}]
	 Recall that each measure $\hat{\mu}_N$ is   concentrated on $\LNB{C_3}\subset\td\times\overline{\mathbb{B}}_{C_3}$. Furthermore, the set $\mathcal{P}(\td\times\overline{\mathbb{B}}_{C_3})$ is compact in the topology of narrow convergence. Now, to obtain the conclusion of the corollary, it suffices to notice that on $\mathcal{P}(\td\times\overline{\mathbb{B}}_{C_3})$ the narrow convergence coincides with the $(C_L^0)^*$-weak convergence and use Theorem~\ref{mather:th:mather_limit}.
\end{proof}

\bibliography{KAM_Markov_a}

\end{document}